\documentclass[10pt]{article}

\usepackage{amsfonts} 
\usepackage{amsmath,hhline,latexsym,mathrsfs,amsthm}
\usepackage{amssymb}
\usepackage{relsize}
\DeclareMathAlphabet\mathbfcal{OMS}{cmsy}{b}{n}

\usepackage[latin1]{inputenc}

\usepackage{enumitem}

\usepackage{hyperref}

\usepackage{color}
\usepackage{graphicx}
\usepackage{comment}

\newtheorem{theorem}{Theorem}
\newtheorem{prop}{Proposition}
\newtheorem{lemma}{Lemma}

\newtheorem{remark}{Remark}

\newcommand{\R}{\mathbb{R}}
\newcommand{\N}{\mathbb{N}}
\newcommand{\A}{\mathcal{A}}

\begin{document}

\title{Constructive proof of the exact controllability for semi-linear wave equations}

\author{
\textsc{J\'er\^ome Lemoine}\thanks{Laboratoire de Math\'ematiques Blaise Pascal, Universit\'e Clermont Auvergne, 
UMR CNRS 6620, Campus universitaire des C\'ezeaux, 3, place Vasarely, 63178, Aubi\`ere, France. E-mail: {\tt jerome.lemoine@uca.fr.}}\and
\textsc{Arnaud M\"unch}\thanks{Laboratoire de Math\'ematiques Blaise Pascal, Universit\'e Clermont Auvergne, 
UMR CNRS 6620, Campus universitaire des C\'ezeaux, 3, place Vasarely, 63178, Aubi\`ere, France. E-mail: {\tt arnaud.munch@uca.fr.}}
}

%\date{}
\maketitle

\begin{abstract}
The exact distributed controllability of the semilinear wave equation $\partial_{tt}y-\Delta y + g(y)=f \,1_{\omega}$ posed over multi-dimensional and bounded domains, assuming that $g\in C^1(\mathbb{R})$ satisfies the growth condition $\limsup_{r\to \infty} g(r)/(\vert r\vert \ln^{1/2}\vert r\vert)=0$ has been obtained by Fu, Yong and Zhang in 2007. The proof based on a non constructive Leray-Schauder fixed point theorem makes use of precise estimates of the observability constant for a linearized wave equation. 
Assuming that $g^\prime$ does not grow faster than $\beta \ln^{1/2}\vert r\vert$ at infinity for $\beta>0$ small enough and that $g^\prime$ is uniformly H\"older continuous on $\mathbb{R}$ with exponent $s\in (0,1]$, we design a constructive proof yielding an explicit sequence converging to a controlled solution  for the semilinear equation, at least with order $1+s$ after a finite number of iterations. %Numerical experiments in the two dimensional setting illustrate our analysis. 
\end{abstract}



\section{Introduction}

 Let $\Omega$ be a bounded domain of $\mathbb{R}^d$, $d\in \{2,3\}$ with $C^{1,1}$ boundary and $\omega \subset\subset \Omega$ be a non empty open set. Let $T>0$ and denote $Q_T:=\Omega\times (0,T)$, $q_T:=\omega\times (0,T)$ and $\Sigma_T:=\partial\Omega\times (0,T)$. We consider the semilinear wave equation
\begin{equation}
\label{eq:wave-NL}
\left\{
\begin{aligned}
& \partial_{tt}y - \Delta y +  g(y)= f 1_{\omega}, & \textrm{in}\,\, Q_T,\\
& y=0, & \textrm{on}\,\, \Sigma_T, \\
& (y(\cdot,0),y_t(\cdot,0))=(u_0,u_1), & \textrm{in}\,\, \Omega,
\end{aligned}
\right.
\end{equation}
where $(u_0,u_1)\in \boldsymbol{V}:=H^1_0(\Omega)\times L^2(\Omega)$ is the initial state of $y$ and $f\in L^2(q_T)$ is a {\it control} function. Here and throughout the paper, $g:\mathbb{R}\to \mathbb{R}$ is a function of class $C^1$ such that
$\vert g(r)\vert \leq C(1+\vert r\vert)\ln(2+\vert r\vert)$ for every $r\in \mathbb{R}$  and some $C>0$. Then, \eqref{eq:wave-NL} has a unique global weak solution in $C([0,T]; H_0^1(\Omega)) \cap C^1([0,T]; L^2(\Omega))$ (see \cite{CazenaveHaraux1980},\cite{Cannarsa_Loreti_Komornik2002}).

The exact controllability for \eqref{eq:wave-NL} in time $T$ is formulated as follows: for any $(u_0,u_1), (z_0,z_1)\in \boldsymbol{V}$, find a control function $f\in L^2(q_T)$ such that the weak solution of \eqref{eq:wave-NL} satisfies $(y(\cdot,T),\partial_{t}y(\cdot,T))=(z_0,z_1)$. 
Assuming a growth condition on the nonlinearity $g$ at infinity, this problem has been solved in \cite{Zhang2007}.
\begin{theorem}\label{ZhangTH}\cite{Zhang2007} 
For any $x_0\in \mathbb{R}^d\backslash\overline{\Omega}$, let $\Gamma_0=\{x\in \partial\Omega, (x-x_0)\cdot \nu(x)>0\}$ and, for any $\epsilon>0$,  $\mathcal{O}_{\epsilon}(\Gamma_0)=\{y\in \mathbb{R}^d \mid \vert y-x\vert \leq \epsilon \, \textrm{for}\,  x\in \Gamma_0\}$. Assume 
\begin{enumerate}[label=$\bf (H_0)$,leftmargin=1.5cm]
\item\label{condgeom}\ $T>2\max_{x\in\overline{\Omega}}\vert x-x_0\vert$ and  $\omega\subseteq \mathcal{O}_{\epsilon}(\Gamma_0)\cap \Omega$ for some $\epsilon>0$.
\end{enumerate}
If $g$ satisfies 
\begin{enumerate}[label=$\bf (H_1)$,leftmargin=1.5cm]
\item\label{asymptotic_behavior}\ $\limsup_{\vert r\vert \to \infty} \frac{\vert g(r)\vert }{\vert r\vert \ln^{1/2}\vert r\vert}=0$% \label{asymptotic_behavior}
\end{enumerate}
then \eqref{eq:wave-NL} is exactly controllable in time $T$.
\end{theorem}

This result improves \cite{Li_Zhang_2000} where a stronger condition of the support $\omega$ is made, namely that $\omega$ is a neighborhood of $\partial\Omega$ and that $T>\textrm{diam}(\Omega\backslash\omega)$. In Theorem \ref{ZhangTH}, $\Gamma_0$ is the usual star-shaped part of the whole boundary of $\Omega$ introduced in \cite{JLL88}. 

A special case of Theorem \ref{ZhangTH} is when $g$ is globally Lipschitz continuous, which gives the main result of \cite{zuazua91}, later generalized to an abstract setting in \cite{Lasiecka91} using a global version of Inverse Function Theorem  and improved in \cite{Zhang_2000} for control domains $\omega$ satisfying the classical multiplier method of Lions \cite{JLL88}. 

Theorem \ref{ZhangTH} extends to the multi-dimensional case the result of \cite{zuazua93} devoted to the one dimensional case under the condition  $\limsup_{\vert r\vert \to \infty} \frac{\vert g(r)\vert }{\vert r\vert \ln^2\vert r\vert}=0$, relaxed later on in \cite{Cannarsa_Loreti_Komornik2002}, following \cite{Imanuvilov89}, and in \cite{Martinez2003}. The exact controllability for subcritical nonlinearities is obtained in \cite{DehmanLebeau} assuming the sign condition $rg(r)\geq 0$ for every $r\in \mathbb{R}$. This latter assumption has been weakened in \cite{JolyLaurent2014} to an asymptotic sign condition leading to a semi-global controllability result in the sense that the final data $(z_0,z_1)$  is prescribed in a precise subset of 
$\boldsymbol{V}$. In this respect, we also mention in the one dimensional case \cite{coron-trelat-wave-06} where a positive boundary controllability result is proved for a steady-state initial and final data  specific class of initial and final data and for $T$ large enough by a quasi-static deformation approach. 

The proof given in \cite{Li_Zhang_2000,Zhang2007} is based on a fixed point argument introduced in \cite{zuazua90,zuazua93} that reduces the exact controllability problem to the obtention of suitable \textit{a priori} estimates for the linearized wave equation with a potential (see Proposition \ref{iobs_zhang_2007} in appendix A). More precisely, it is shown that the operator $K:L^\infty(0,T;L^d(\Omega))\to L^\infty(0,T;L^d(\Omega))$ where $y_\xi:=K(\xi)$ is a controlled solution through the control function $f_{\xi}$ of the linear boundary value problem 
\begin{equation}
\label{NL_z}
\left\{
\begin{aligned}
& \partial_{tt} y_{\xi} - \Delta y_{\xi} +  y_\xi \,\widehat{g}(\xi)= -g(0)+f_\xi 1_{\omega}, &\textrm{in}\,\, Q_T,\\
& y_\xi=0, &\textrm{on}\,\, \Sigma_T, \\
& (y_\xi(\cdot,0),\partial_{t} y_{\xi}(\cdot,0))=(u_0,u_1), &\textrm{in}\,\, \Omega,
\end{aligned}
\right.
\qquad
\widehat{g}(r):=
\left\{ 
\begin{aligned}
& \frac{g(r)-g(0)}{r} & r\neq 0,\\
& g^{\prime}(0) & r=0
\end{aligned}
\right., 
\end{equation}
satisfying $(y_{\xi}(\cdot,T),y_{\xi,t}(\cdot,T))=(z_0,z_1)$ has a fixed point. The control $f_{\xi}$ is chosen in \cite{Li_Zhang_2000} as the one of minimal $L^2(q_T)$-norm. The existence of a fixed point for the compact operator $K$ is obtained by using the Leray-Schauder's degree theorem. Precisely, it is shown under the growth assumption \ref{asymptotic_behavior} that there exists a constant $M=M(\Vert u_0,u_1\Vert_{\boldsymbol{V}}, \Vert z_0,z_1\Vert_{\boldsymbol{V}}$) such that $K$ maps the ball $B_{L^\infty(0,T;L^d(\Omega))}(0,M)$ into itself.

 The main goal of this article is to design an algorithm providing an explicit sequence $(f_k)_{k\in\mathbb{N}}$ that converges strongly to an exact control for~(\ref{eq:wave-NL}). A first idea that comes to mind is to consider the Picard iterations $(y_k)_{k\in \mathbb{N}}$ associated with the operator $K$ defined by $y_{k+1}=K(y_k)$, $k\geq 0$ initialized with any element $y_0\in L^{\infty}(0,T;L^d(\Omega))$. The resulting sequence of controls $(f_k)_{k\in \mathbb{N}}$ is then so that $f_{k+1}\in L^2(q_T)$ is the control of minimal $L^2(q_T)$ norm for $y_{k+1}$ solution of 
\begin{equation}
\label{NL_z_k}
\left\{
\begin{aligned}
& \partial_{tt} y_{k+1} - \Delta y_{k+1} +  y_{k+1} \,\widehat{g}(y_k)= -g(0)+f_{k+1} 1_{\omega}, & \textrm{in}\quad Q_T,\\
& y_{k+1}=0, & \textrm{on}\,\, \Sigma_T, \\
& (y_{k+1}(\cdot,0),\partial_{t} y_{k+1}(\cdot,0))=(y_0,y_1), & \textrm{in}\,\, \Omega.
\end{aligned}
\right. 
\end{equation}
Such a strategy usually fails since the operator $K$ is in general not contracting, even if $g$ is globally Lipschitz. We refer to \cite{EFC-AM-2012} providing numerical evidence of the lack of convergence in parabolic cases (see also Remark \ref{rem_contract} in appendix A). A second idea is to use a Newton type method in order to find a zero of the $C^1$ mapping $\widetilde{F}: Y \mapsto W$ defined by 
   \begin{equation}\label{def-F}
\widetilde{F}(y,f):= \biggl(\partial_{tt} y - \Delta y  + g(y) - f 1_{\omega}, y(\cdot\,,0) - u_0, \partial_{t}y(\cdot\,,0) - u_1, y(\cdot\,,T) -z_0, \partial_{t}y(\cdot\,,T) - z_1\biggr) 
   \end{equation}
   for some appropriates Hilbert spaces $Y$ and $W$ (see further): given $(y_0,f_0)$ in $Y$, the sequence $(y_k,f_k)_{k\in \mathbb{N}}$ is defined iteratively by $(y_{k+1},f_{k+1})=(y_k,f_k)-(Y_k,F_k)$ where $F_k$ is a control for $Y_k$ solution of  
\begin{equation}
\label{Newton-nn}
\left\{
\begin{aligned}
&\partial_{tt}Y_{k}-\Delta Y_k +  g^{\prime}(y_k)\,Y_{k} = F_{k}\, 1_{\omega} + \partial_{tt} y_{k} - \Delta y_k  + g(y_k) - f_k 1_{\omega},
                                                                                    & \quad  \textrm{in}\quad Q_T,\\
& Y_k=0,                                                              & \quad  \textrm{on}\quad \Sigma_T,\\ 
& Y_k(\cdot,0)=u_0-y_k(\cdot,0),  \partial_{t}Y_{k}(\cdot,0)=u_1-\partial_{t} y_{k}(\cdot,0)                                   & \quad  \textrm{in}\quad \Omega
\end{aligned}
\right.
   \end{equation}
   such that $Y_k(\cdot,T)=-y_k(\cdot,T)$ and $\partial_{t}Y_{k}(\cdot,T)=-\partial_{t}y_{k}(\cdot,T)$ in $\Omega$.
This linearization makes appear an operator $K_N$, so that $y_{k+1}=K_N(y_k)$ involving the first derivative of $g$. 
 However, as it is well known, such a sequence may fail to converge if the initial guess $(y_0,f_0)$ is not close enough to a zero of $F$ (see \cite{EFC-AM-2012} where divergence is observed numerically for large data).
  
  The controllability of nonlinear partial differential equations has attracted a large number of works in the last decades (see the monography \cite{Coron-Book-07} and references therein).
However, as far as we know, few are concerned with the approximation of exact controls for nonlinear partial differential equations, and the construction of convergent control approximations for controllable nonlinear equations remains a challenge. 

In this article, given any initial data $(u_0,u_1)\in \boldsymbol{V}$, we design an algorithm providing a sequence $(y_k,f_k)_{k\in \N}$ converging to a controlled pair for \eqref{eq:wave-NL}, under assumptions on $g$ that are slightly stronger than the one done in Theorem \ref{ZhangTH}. Moreover, after a finite number of iterations, the convergence is super-linear. This is done by introducing a quadratic functional measuring how much a pair $(y,f) \in Y$ is close to a controlled solution for \eqref{eq:wave-NL} and then by determining a particular minimizing sequence enjoying the announced property. A natural example of an error (or least-squares) functional is given by $\widetilde{E}(y,f):=\frac{1}{2}\Vert \widetilde{F}(y,f)\Vert^2_W$ to be minimized over $Y$. Exact controllability for~(\ref{eq:wave-NL}) is reflected by the fact that the global minimum of the nonnegative functional $\widetilde{E}$ is zero, over all pairs $(y,f)\in Y$ solutions of \eqref{eq:wave-NL}. In the line of recent works on the Navier-Stokes system (see \cite{lemoinemunch_time, lemoinemunch_NUMER}), we determine, using an appropriate descent direction, a minimizing sequence $(y_k,f_k)_{k\in \mathbb{N}}$ converging to a zero of the quadratic functional.

The paper is organized as follows.  In Section \ref{sec:LS}, we define the (nonconvex) least-squares functional $E$ and the corresponding (nonconvex) optimization problem \eqref{extremal_problem}. We show that $E$ is Gateaux-differentiable and that any critical point $(y,f)$ for $E$ such that $g^\prime(y)\in L^\infty(0,T;L^d(\Omega))$ is also a zero of $E$. This is done by introducing an adequate descent direction $(Y^1,F^1)$ for $E$ at any $(y,f)$ for which $E^\prime(y,f)\cdot (Y^1,F^1)$ is proportional to $\sqrt{E(y,f)}$. This instrumental fact compensates the failure of convexity of $E$ and is at the base of the global convergence properties of the least-squares algorithm. The design of this algorithm is done by determining a minimizing sequence based on $(Y^1,F^1)$, which is proved to converge to a controlled pair for the semilinear wave equation \eqref{eq:wave-NL}, in our main result (Theorem \ref{main_theorem}), under appropriate assumptions on $g$. Moreover, we prove that, after a finite number of iterations, the convergence is super-linear.
Theorem \ref{main_theorem} is proved in Section \ref{sec:convergence}.
We show in Section \ref{sec:remarks} that our least-squares approach coincides with the classical \emph{damped Newton method} applied to a mapping similar to $\widetilde{F}$, and we give a number of other comments. 
In Appendix \ref{sec:linearizedwave}, we state some \textit{a priori} estimates for the linearized wave equation with potential in $L^\infty(0,T;L^d(\Omega))$ and source term in $L^2(Q_T)$ and we show that the operator $K$ is contracting if $\Vert \hat{g}^\prime\Vert_{L^\infty(\R)}$ is small enough.

As far as we know, the method introduced and analyzed in this work is the first one providing an explicit, algorithmic construction of exact controls for semilinear wave equations with non Lipschitz nonlinearity and defined over multi-dimensional bounded domains. It extends the one-dimensional study addressed in \cite{munch_trelat}. For parabolic equations with Lipschitz nonlinearity, we mention \cite{lemoine_gayte_munch}. These works devoted to controllability problems takes their roots in earlier works, namely \cite{lemoinemunch_time, lemoinemunch_NUMER}, concerned with the approximation of solution of Navier-Stokes type problem, through least-square methods: they refine the analysis performed in \cite{lemoine-Munch-Pedregal-AMO-20, MunchMCSS2015} inspired from the seminal contribution \cite{Bristeau1979}.

\paragraph{Notations.}
Throughout, we denote by $\Vert \cdot \Vert_{\infty} $ the usual norm in $L^\infty(\R)$, by $(\cdot,\cdot)_X$ the scalar product of $X$ (if $X$ is a Hilbert space) and by $\langle \cdot, \cdot \rangle_{X,Y}$ the duality product between $X$ and $Y$.
The notation $\Vert \cdot\Vert_{2,q_T}$ stands for $\Vert \cdot\Vert_{L^2(q_T)}$ and $\Vert \cdot\Vert_p$ for $\Vert \cdot\Vert_{L^p(Q_T)}$, $p\in \mathbb{N}^\star$.

Given any $s\in [0,1]$, we introduce for any $g\in C^1(\R)$ the following hypothesis : 
\begin{enumerate}[label=$\bf (\overline{H}_s)$,leftmargin=1.5cm]
\item\label{constraint_g_holder}\ $[g^\prime]_s := \sup_{a,b\in \R \atop a\neq b} \frac{\vert g^\prime(a)-g^\prime(b)\vert}{\vert a-b\vert^s}  < +\infty$
\end{enumerate}
meaning that $g^\prime$ is uniformly H\"older continuous with exponent $s$.
For $s=0$, by extension, we set $[g']_0 :=2\Vert g^\prime\Vert_\infty$.
In particular, $g$ satisfies $\bf (\overline{H}_0)$ if and only if $g\in \mathcal{C}^1(\R)$ and $g^\prime\in L^\infty(\R)$, and $g$ satisfies $\bf (\overline{H}_1)$ if and only if $g^\prime$ is Lipschitz continuous (in this case, $g^\prime$ is almost everywhere differentiable and $g^{\prime\prime}\in L^\infty(\R)$, and we have $[g^\prime]_s\leq \Vert g^{\prime\prime}\Vert_\infty$).

We also denote by $C$ a positive constant depending only on $\Omega$ and $T$ that may vary from lines to lines. 

In the rest of the paper, we assume that the open set $\omega$ and the time $T$ satisfy \ref{condgeom}.

\section{The least-squares functional and its properties }\label{sec:LS}

\subsection{The least-squares problem}

We define the Hilbert space $\mathcal{H}$
\begin{equation}
\nonumber
\mathcal{H}=
\biggl\{(y,f)\in L^2(Q_T)\times L^2(q_T) \mid \partial_{tt}y - \Delta y\in L^2(Q_T), \, (y(\cdot,0),\partial_{t}y(\cdot,0))\in \boldsymbol{V}, \, y=0 \,\textrm{on}\, \Sigma_T\biggr\}
\end{equation}
endowed with the scalar product

\begin{equation}\nonumber
\begin{aligned}
((y,f),(\overline{y},\overline{f}))_{\mathcal{H}}= (y,\overline{y})_2+ \big((y(\cdot,0),\partial_{t}y(\cdot,0)),&(\overline{y}(\cdot,0),\partial_{t}\overline{y}(\cdot,0))\big)_{\boldsymbol{V}}\\
&+(\partial_{tt}y - \Delta y,\partial_{tt}\overline{y} - \Delta\overline{y})_2+ (f,\overline{f})_{2,q_T}
\end{aligned}
\end{equation}
and the norm $\Vert (y,f)\Vert_{\mathcal{H}}:=\sqrt{((y,f),(y,f))_{\mathcal{H}}}$. 
Then, for any $(u_0,u_1), (z_1,z_1)\in \boldsymbol{V}$, we define the subspaces of $\mathcal{H}$
\begin{equation}
\nonumber
\begin{aligned}
& \mathcal{A}=\biggl\{(y,f)\in \mathcal{H} \mid (y(\cdot,0),\partial_{t}y(\cdot,0))=(u_0,u_1),\, (y(\cdot,T),\partial_{t}y(\cdot,T))=(z_0,z_1)\biggr\},\\
& \mathcal{A}_0=\biggl\{(y,f)\in \mathcal{H} \mid (y(\cdot,0),\partial_{t}y(\cdot,0))=(0,0),\, (y(\cdot,T),\partial_{t}y(\cdot,T))=(0,0)\biggr\}.
\end{aligned}
\end{equation}
\par\noindent
We consider the following non convex extremal problem :
\begin{equation}
\label{extremal_problem}
\inf_{(y,f)\in \mathcal{A}} E(y,f), \qquad E(y,f):=\frac{1}{2}\big\Vert \partial_{tt}y-\Delta y + g(y)-f\,1_{\omega}\big\Vert^2_2
\end{equation}
justifying the least-squares terminology we have used. Remark that we can write $\mathcal{A}=(\overline{y},\overline{f})+\mathcal{A}_0$ for any element $(\overline{y},\overline{f})\in \mathcal{A}$. The problem is therefore equivalent to the minimization of $E(\overline{y}+y,f+\overline{f})$ over $\mathcal{A}_0$ for any $(\overline{y},\overline{f})\in \mathcal{A}$.

The functional $E$ is well-defined in $\mathcal{A}$. Precisely, 
\begin{lemma}
There exists a positive constant $C>0$ such that $E(y,f)\leq C \Vert (y,f)\Vert^3_{\mathcal{H}}$
for any $(y,f)\in \mathcal{A}$.
\end{lemma}
\begin{proof}
\textit{A priori} estimate for the linear wave equation reads as  
$$
\Vert (y,\partial_{t}y)\Vert^2_{L^{\infty}(0,T; \boldsymbol{V})} \leq C \biggl( \Vert \partial_{tt}y-\Delta y\Vert^2_2+\Vert u_0,u_1\Vert^2_{\boldsymbol{V}}\biggr)
$$
for any $y$ such that $(y,f)\in \mathcal{A}$. Using that $\vert g(r)\vert \leq C(1+\vert r\vert)\log(2+\vert r\vert)$ for every $r\in \mathbb{R}$  and some $C>0$, we infer that 
$$
\begin{aligned}
\Vert g(y)\Vert_2^2 &\leq C^2 \int_{Q_T}  \biggl((1+\vert y\vert) \log(2+\vert y\vert)\biggr)^2\\
&\leq C^2 \int_{Q_T} (1+\vert y\vert)^3\leq C^2(\vert Q_T\vert^3+ \Vert y\Vert^3_{L^3(Q_T)})\\
& \leq C^2\big(\vert Q_T\vert^3+ \Vert y\Vert^3_{L^\infty(0,T;H_0^1(\Omega))}\big)
\end{aligned}
$$
for which we get $
\begin{aligned}
E(y,f)\leq  C\big(\Vert \partial_{tt}y-\Delta y\Vert_2^2+ \Vert f\Vert_{2,q_T}^2+ \vert Q_T\vert^3+ \Vert y\Vert^3_{L^\infty(0,T;H_0^1(\Omega))} \big)
\end{aligned}
$
leading to the result.
\end{proof}

Within the hypotheses of Theorem \ref{ZhangTH}, the infimum of the functional of $E$ is zero and is reached by at least one pair $(y,f)\in \mathcal{A}$, solution of  (\ref{eq:wave-NL}) and satisfying $(y(\cdot,T),\partial_t y(\cdot,T))=(z_0,z_1)$. Conversely, any pair $(y,f)\in \mathcal{A}$ for which $E(y,f)$ vanishes is solution of (\ref{eq:wave-NL}). In this sense, the functional  $E$ is an \textit{error functional} which measures the deviation of $(y,f)$ from being a solution of the underlying nonlinear equation.  A practical way of taking a functional to its minimum is through the  use of gradient descent directions. In doing so, the presence of local minima is always something that  may dramatically spoil the whole scheme. The unique structural property that discards this possibility is the convexity of the functional $E$. However, for nonlinear equation like (\ref{eq:wave-NL}), one cannot expect this property to hold for the functional $E$. Nevertheless, we are going to construct a minimizing sequence which always converges to a zero of $E$.

In order to construct such minimizing sequence, we formally look, for any $(y,f)\in \mathcal{A}$, for a pair $(Y^1,F^1)\in \mathcal{A}_0$ solution of the following formulation 

\begin{equation}
\label{wave-Y1}
\left\{
\begin{aligned}
& \partial_{tt}Y^1 - \Delta Y^1 +  g^{\prime}(y)\cdot Y^1 = F^1 1_{\omega}+\big(\partial_{tt}y-\Delta y+ g(y)-f\, 1_{\omega}\big), &\textrm{in}\,\, Q_T,\\
& Y^1=0, &\textrm{on}\,\, \Sigma_T, \\
& (Y^1(\cdot,0),\partial_{t}Y^1(\cdot,0))=(0,0), & \textrm{in}\,\, \Omega.
\end{aligned}
\right.
\end{equation}
Since $(Y^1,F^1)$ belongs to $\mathcal{A}_0$, $F^1$
is a null control for $Y^1$. Among the controls of this linear equation, we select the control of minimal $L^2(q_T)$ norm. In the sequel, we shall call the corresponding solution $(Y^1,F^1)\in \mathcal{A}_0$ the solution of minimal control norm.  We have the following property. 

\begin{prop}\label{estimateC1C2} For any $(y,f)\in \mathcal{A}$, there exists a pair $(Y^1,F^1)\in \mathcal{A}_0$ solution of (\ref{wave-Y1}). Moreover, the pair $(Y^1,F^1)$ of minimal control norm satisfies the following estimates :  

\begin{equation} \label{estimateF1Y1}
\Vert (Y^1,\partial_{t}Y^1)\Vert_{L^{\infty}(0,T;\boldsymbol{V})}+ \Vert F^1\Vert_{2,q_T} \leq C  e^{C \Vert g^{\prime}(y)\Vert^2_{L^\infty(0,T;L^d(\Omega))} }\sqrt{E(y,f)},
\end{equation}

and 

\begin{equation} \label{estimateF1Y1_A0}
\Vert (Y^1,F^1)\Vert_{\mathcal{H}} \leq C  \big(1+\Vert g^\prime(y)\Vert_{L^\infty(0,T;L^3(\Omega))}\big)e^{C \Vert g^{\prime}(y)\Vert^2_{L^\infty(0,T;L^d(\Omega))} }\sqrt{E(y,f)}
\end{equation}
 for some positive constant $C>0$.
\end{prop}
\begin{proof} The first estimate is a consequence of Proposition \ref{controllability_result} using the equality $\Vert \partial_{tt} y-\Delta y +g(y)-f\,1_{\omega}\Vert_2=\sqrt{2E(y,f)}$. The second one follows from 
\begin{equation}
\nonumber
\begin{aligned}
\Vert (Y^1,F^1)\Vert_{\mathcal{H}} & \leq  \Vert \partial_{tt}Y^1-\Delta Y^1\Vert_2 + \Vert Y^1\Vert_2+\Vert F^1\Vert_{2,q_T}+\Vert Y^1(\cdot,0),\partial_{t}Y^1(\cdot,0)\Vert_{\boldsymbol{V}} \\
& \leq \Vert Y^1\Vert_2 + \Vert g^{\prime}(y)Y^1\Vert_2 + 2\Vert F^1\Vert_{2,q_T}+\sqrt{2}\sqrt{E(y,f)}\\
& \leq C \big(1+\Vert g^{\prime}(y)\Vert_{L^\infty(0,T;L^3(\Omega))}\big)e^{C \Vert g^{\prime}(y)\Vert^2_{L^\infty(0,T;L^d(\Omega))} }\sqrt{E(y,f)}
\end{aligned}
\end{equation}
using that 
$$
\begin{aligned}
\Vert g^{\prime}(y)Y^1\Vert^2_2 \leq & \int_0^T \Vert g^\prime(y)\Vert^2_{L^3(\Omega)}  \Vert Y^1\Vert^2_{L^6(\Omega)}\\
&\leq \Vert g^\prime(y)\Vert^2_{L^\infty(0,T;L^3(\Omega))} \Vert Y^1\Vert^2_{L^\infty(0,T; L^6(\Omega))}\\
&\leq C\Vert g^\prime(y)\Vert^2_{L^\infty(0,T;L^3(\Omega))} \Vert Y^1\Vert^2_{L^\infty(0,T; H_0^1(\Omega))}.
\end{aligned}
 $$
\end{proof}

\subsection{Main properties of the functional $E$}

The interest of the pair $(Y^1,F^1)\in \mathcal{A}_0$ lies in the following result.
\begin{prop}\label{differentiabiliteE}
Assume that $g$ satisfies \ref{constraint_g_holder} for some $s\in [0,1]$. Let $(y,f)\in \mathcal{A}$  and let $(Y^1,F^1)\in \mathcal{A}_0$ be a solution of (\ref{wave-Y1}). Then the derivative of $E$ at the point $(y,f)\in \mathcal{A}$ along the direction $(Y^1,F^1)$ satisfies 
\begin{equation}\label{estimateEEprime}
E^{\prime}(y,f)\cdot (Y^1,F^1)=2E(y,f).
\end{equation}
\end{prop} 
\begin{proof} We preliminary check that for all $(Y,F)\in \mathcal{A}_0$ the functional $E$ is differentiable at the point $(y,f)\in \mathcal{A}$ along the direction $(Y,F)\in \mathcal{A}_0$. For any $\lambda\in \mathbb{R}$, simple computations lead to the equality 
$$
\begin{aligned}
E(y+\lambda Y,f+\lambda F) =E(y,f)+ \lambda E^{\prime}(y,f)\cdot (Y,F) + h((y,f),\lambda (Y,F))
\end{aligned}
$$
with 
\begin{equation}\label{Efirst}
E^{\prime}(y,f)\cdot (Y,F):=\big(\partial_{tt}y-\Delta y+ g(y)-f\, 1_\omega,  \partial_{tt}Y-\Delta y+ g^\prime(y)Y-F\, 1_\omega\big)_2
\end{equation}
and
$$
\begin{aligned}
h((y,f),\lambda (Y,F)):=& \frac{\lambda^2}{2}\big(\partial_{tt}Y-\Delta Y+ g^\prime(y)Y-F\, 1_\omega,\partial_{tt}Y-\Delta Y+ g^\prime(y)Y-F\, 1_\omega\big)_2\\
& +\lambda \big(\partial_{tt}Y-\Delta Y+ g^\prime(y)Y-F\, 1_\omega,l(y,\lambda Y)\big)_2\\
& +\big(\partial_{tt}y-\Delta y+ g(y)-f\, 1_\omega,l(y,\lambda Y)\big)+ \frac{1}{2}(l(y,\lambda Y),l(y,\lambda Y))
\end{aligned}
$$ 
where $l(y,\lambda Y):=g(y+\lambda Y)-g(y)-\lambda g^{\prime}(y)Y$.
The application $(Y,F)\to E^{\prime}(y,f)\cdot (Y,F)$ is linear and continuous from $\mathcal{A}_0$ to $\mathbb{R}$ as it satisfies
\begin{equation}
\label{useful_estimate}
\begin{aligned}
\vert E^{\prime}(y,f)\cdot (Y,F)\vert &\leq  \Vert \partial_{tt}y-\Delta y+ g(y)-f\, 1_\omega\Vert_2 \Vert \partial_{tt}Y-\Delta Y+ g^\prime(y)Y-F\, 1_\omega\Vert_2\\
& \leq \sqrt{2E(y,f)} \biggl(\Vert (\partial_{tt}Y-\Delta Y)\Vert_2 + \Vert g^\prime(y)\Vert_{L^\infty(0,T;L^3(\Omega))} \Vert Y^1\Vert_{L^\infty(0,T; H_0^1(\Omega))} + \Vert F\Vert_{2,q_T}\biggr)\\
& \leq \sqrt{2E(y,f)} \max\big(1,\Vert g^\prime(y)\Vert_{L^\infty(0,T;L^3(\Omega))}\big) \Vert (Y,F)\Vert_{\mathcal{H}}.
\end{aligned}
\end{equation}
Similarly, for all $\lambda\in \mathbb{R}^\star$, 
$$
\begin{aligned}
\biggl\vert \frac{1}{\lambda} h((y,f),\lambda (Y,F))&\biggr\vert  \leq \frac{\lambda}{2} \Vert \partial_{tt}Y-\Delta Y+ g^\prime(y)Y-F\, 1_\omega\Vert^2_2 \\
& +\biggl(\lambda\Vert \partial_{tt}Y-\Delta Y+ g^\prime(y)Y-F\, 1_\omega\Vert_2 +\sqrt{2E(y,f)}+\frac{1}{2}\Vert l(y,\lambda Y)\Vert_2\biggl)\frac{1}{\lambda}\Vert l(y,\lambda Y)\Vert_2.
\end{aligned}
$$
For any $(x,y)\in \R^2$ and $\lambda\in \R$, we then write $g(x+\lambda y)-g(x)=\int_0^\lambda y g'(x+\xi y)d\xi$ leading to 
$$
\begin{aligned}
|g(x+\lambda y)-g(x)-\lambda g'(x)y|
&\le \int_0^\lambda |y| |g'(x+\xi y)-g'(x)|d\xi\\
&\le \int_0^\lambda |y|^{1+s}|\xi|^s\frac{|g'(x+\xi y)-g'(x)|}{|\xi y|^s}d\xi\le [g^\prime]_s|y|^{1+s}\frac{\vert\lambda\vert^{1+s}}{1+s}.
\end{aligned}
$$
It follows that 
$$
|l(y,\lambda Y)|=|g(y+\lambda Y)-g(y)-\lambda g^{\prime}(y) Y|\le  [g^\prime]_s\frac{\vert \lambda\vert^{1+s}}{1+s}|Y|^{1+s}
$$
and 
\begin{equation}
\label{estimLY}
\begin{aligned}
\frac{1}{\vert \lambda\vert}\bigl \Vert l(y,\lambda Y)\bigr\Vert_2
&\le [g^\prime]_s\frac{\vert\lambda\vert^s}{1+s}\bigl \Vert |Y|^{1+s}\bigr\Vert_2.
\end{aligned}
 \end{equation}
 But 
 $
 \bigl \Vert |Y|^{1+s}\bigr\Vert^2_2=\Vert Y\Vert_{2(s+1)}^{2(s+1)}\leq C \Vert Y\Vert_{L^{\infty}(0,T;L^4(\Omega))}^{2(s+1)}.
$
Consequently, for $s>0$, $\vert \frac{1}{\lambda}\vert \Vert l(y,\lambda Y)\Vert_2\to 0$ as $\lambda\to 0$ and $\vert h((y,f),\lambda (Y,F))\vert=o(\lambda)$. In the case $s=0$ leading to $g^\prime\in L^\infty(\R)$, the result follows from the Lebesgue dominated convergence theorem:  we have
$$
\Big\vert \frac{1}{\lambda}\ell(y,\lambda Y)\Big\vert = \Big\vert \frac{g(y+\lambda Y)-g(y)}{\lambda}-g^{\prime}(y)Y\Big\vert\leq 2\Vert g^{\prime}\Vert_{\infty} \vert Y\vert, \quad a.e. \, \textrm{in}\, Q_T
$$
and $\big\vert \frac{1}{\lambda}\ell(y,\lambda Y)\big\vert = \big\vert \frac{g(y+\lambda Y)-g(y)}{\lambda}-g^{\prime}(y)Y\big\vert\to 0$ as $\lambda\to 0$ a.e. in $Q_T$. It follows that $\vert \frac{1}{\lambda}\vert \Vert \ell(y,\lambda Y)\Vert_2\to 0$ as $\lambda\to 0$ as well. We deduce that the functional $E$ is differentiable at the point $(y,f)\in \mathcal{A}$ along the direction $(Y,F)\in \mathcal{A}_0$.

Eventually, the equality (\ref{estimateEEprime}) follows from the definition of the pair $(Y^1,F^1)$ given in (\ref{wave-Y1}). 
\end{proof}

%\begin{remark}
%In our one dimensional setting, we can actually show that the functional $E$ is differentiable, using that the convergence in $\mathcal{A}$ implies the pointwise convergence. 
%\end{remark}

\vskip 0.25cm
Remark that from the equality \eqref{Efirst}, the derivative $E^{\prime}(y,f)$ is independent of $(Y,F)$. We can then define the norm $\Vert E^{\prime}(y,f)\Vert_{\mathcal{A}_0^{\prime}}:= \sup_{(Y,F)\in \mathcal{A}_0\backslash\{0\}} \frac{E^{\prime}(y,f)\cdot (Y,F)}{\Vert (Y,F)\Vert_{\mathcal{H}}}$ associated to $\mathcal{A}_0^{\prime}$, the topological dual of $\mathcal{A}_0$.

\vskip 0.25cm

Combining the equality \eqref{estimateEEprime} and the inequality \eqref{estimateF1Y1}, we deduce the following estimate of $E(y,f)$ in term of the norm of $E^\prime(y,f)$. 
\begin{prop} For any $(y,f)\in\mathcal{A}$, the following inequalities hold true:
\begin{equation}\label{ineq_E_Eprime}
\begin{aligned}
\frac{1}{\sqrt{2}\max\big(1,\Vert g^\prime(y)\Vert_{L^\infty(0,T;L^3(\Omega))}\big)}&\Vert E^{\prime}(y,f)\Vert_{\mathcal{A}_0^\prime}\leq \sqrt{E(y,f)}\\
&\leq \frac{1}{\sqrt{2}}C \biggl(1+\Vert g^\prime(y)\Vert_{L^\infty(0,T;L^3(\Omega))}\biggr)e^{C\Vert g^{\prime}(y)\Vert^2_{L^\infty(0,T;L^d(\Omega))}} \Vert E^{\prime}(y,f)\Vert_{\mathcal{A}_0^\prime}
\end{aligned}
\end{equation}
where $C$ is the positive constant from Proposition \ref{estimateC1C2}.
\end{prop}
\begin{proof} \eqref{estimateEEprime} rewrites
$E(y,f)=\frac12 E^{\prime}(y,f)\cdot (Y^1,F^1)$
where $(Y^1,F^1)\in \mathcal{A}_0$ is solution of (\ref{wave-Y1})  and therefore, with \eqref{estimateF1Y1_A0}
$$
\begin{aligned}
E(y,f)&\le\frac12 \|E'(y,f)\|_{ \mathcal{A}_0'} \|(Y^1,F^1)\|_{ \mathcal{A}_0}\\
& \le \frac12  C \big(1+\Vert g^\prime(y)\Vert_{L^\infty(0,T;L^3(\Omega))}\big)e^{C\Vert g^{\prime}(y)\Vert^2_{L^\infty(0,T;L^d(\Omega))}} \|E'(y,f)\|_{ \mathcal{A}_0'}\sqrt{E(y,f)}.
\end{aligned}
$$
On the other hand, for all $(Y,F)\in\mathcal{A}_0$, the inequality \eqref{useful_estimate}, i.e.
$$
\vert E^{\prime}(y,f)\cdot (Y,F)\vert \leq  \sqrt{2E(y,f)} \max\big(1,\Vert  g^\prime(y)\Vert_{L^\infty(0,T;L^3(\Omega))}\big) \Vert (Y,F)\Vert_{\mathcal{A}_0}
$$
leads to the left inequality.
\end{proof}

Consequently, any \textit{critical} point $(y,f)\in \mathcal{A}$ of $E$ (i.e., $E^\prime(y,f)$ vanishes) such that $\Vert g^\prime(y)\Vert_{L^\infty(0,T;L^3(\Omega))}$ is finite is a zero for $E$, a pair solution of the controllability problem. 
In other words, any sequence $(y_k,f_k)_{k>0}$ satisfying $\Vert E^\prime(y_k,f_k)\Vert_{\mathcal{A}_0^\prime}\to 0$ as $k\to \infty$ and for which $\Vert g^{\prime}(y_k)\Vert_{L^\infty(0,T;L^3(\Omega))}$
is uniformly bounded is such that $E(y_k,f_k)\to 0$ as $k\to \infty$. We insist that this property does not imply the convexity of the functional $E$ (and \textit{a fortiori} the strict convexity of $E$, which actually does not hold here in view of the multiple zeros for $E$) but show that a minimizing sequence for $E$ can not be stuck in a local minimum. %Far from the zeros of $E$, in particular, when $\Vert (y,f)\Vert_{\mathcal{A}}\to \infty$, the inequality also indicates that $E$ tends to be strictly convex. 

On the other hand, the left inequality indicates the functional $E$ is flat around its zero set. As a consequence, gradient-based minimizing sequences may achieve a low speed of convergence (we refer to \cite{AM-PP-2014}
and also \cite{lemoine-Munch-Pedregal-AMO-20} devoted to the Navier-Stokes equation where this phenomenon is observed).  

\

We end this section with the following estimate.

\begin{lemma}\label{lemma_estimateEs} Assume that $g$ satisfies \ref{constraint_g_holder} for some $s\in [0,1]$. For any $(y,f)\in \mathcal{A}$, let $(Y^1,F^1)\in \mathcal{A}_0$ be defined by \eqref{wave-Y1}. For any $\lambda\in \mathbb{R}$  the following estimate holds
\begin{equation}\label{estimateEs}
E\big((y,f)-\lambda (Y^1,F^1)\big)\leq   E(y,f)\biggl(\vert 1-\lambda\vert +\lambda^{1+s}\,c(y) E(y,f)^{s/2}\biggr)^2
\end{equation}
with
$$
c(y):= \frac{C}{(1+s)\sqrt{2}} [g^\prime]_s  d(y)^{1+s}, \quad d(y):=C  e^{C\Vert g^{\prime}(y) \Vert^2_{L^\infty(0,T;L^d(\Omega))}}.
$$
\end{lemma}
\begin{proof} Estimate \eqref{estimLY} applied with $Y=Y^1$ reads  
\begin{equation}
\label{majoration-l}
\begin{aligned}
\bigl \Vert l(y,\lambda Y^1)\bigr\Vert_2
&\le  [g^\prime]_s\frac{\vert\lambda\vert^{1+s}}{1+s}\bigl \Vert |Y^1|^{1+s}\bigr\Vert_2.
\end{aligned}
 \end{equation}
 But $\Vert |Y^1|^{1+s}\bigr\Vert^2_2=\Vert Y^1\Vert_{2(s+1)}^{2(s+1)}\leq C \Vert Y^1\Vert_{L^{\infty}(0,T;H_0^1(\Omega))}^{2(s+1)}$ which together with  \eqref{estimateF1Y1} lead to 
 \begin{equation}
 \label{estim1}
 \bigl \Vert |Y^1|^{1+s}\bigr\Vert_2\leq C \biggl(Ce^{C\Vert g^{\prime}(y)\Vert_{L^\infty(0,T;L^d(\Omega)}}\biggr)^{1+s} E(y,f)^{\frac{1+s}{2}}.
 \end{equation}
 Eventually, we write 
 \begin{equation}
\label{expansionE}
\begin{aligned}
2 &E\big((y,f)-\lambda (Y^1,F^1)\big) \\
&=\biggl\Vert \big(\partial_{tt}y-\Delta y +g(y)-f\,1_\omega\big)-\lambda \big(\partial_{tt}Y^1-\Delta Y^1+g^\prime(y)Y^1-F\,1_\omega\big)+l(y,-\lambda Y^1)\biggr\Vert^2_2\\
& =\biggl\Vert (1-\lambda)\big(\partial_{tt} y-\Delta y+g(y)-f\,1_\omega\big)+l(y,-\lambda Y^1)\biggr\Vert^2_2\\
& \leq \biggl( \big\Vert (1-\lambda)\big(\partial_{tt} y-\Delta y+g(y)-f\,1_\omega\big)\big\Vert_2 +\big\Vert l(y,-\lambda Y^1)\big\Vert_2\biggr)^2\\
&\leq 2 \biggl( \vert 1-\lambda\vert \sqrt{E(y,f)}+[g^\prime]_s\frac{\vert\lambda\vert^{1+s}}{1+s}\bigl \Vert |Y^1|^{1+s}\bigr\Vert_2\biggr)^2\\
& \leq 2 \biggl( \vert 1-\lambda\vert \sqrt{E(y,f)}+[g^\prime]_s\frac{\vert\lambda\vert^{1+s}}{1+s}C \biggl(Ce^{C\Vert g^{\prime}(y)\Vert_{L^\infty(0,T;L^d(\Omega)}}\biggr)^{1+s} E(y,f)^{\frac{1+s}{2}}\biggr)^2
\end{aligned}
\end{equation}
and we get the result. 
\end{proof}

\section{Convergence of a minimizing sequence for $E$}\label{sec:convergence}

\ 
 We now examine the convergence of an appropriate sequence $(y_k,f_k)\in \mathcal{A}$. 
 In this respect, we observe that equality (\ref{estimateEEprime}) shows that $-(Y^1,F^1)$ given by the solution of (\ref{wave-Y1}) is a descent direction for $E$. Therefore, we can define, for any fixed $m\geq 1$, a minimizing sequence $(y_k,f_k)_{k>0}\in \A$  as follows: 
\begin{equation}
\label{algo_LS_Y}
\left\{
\begin{aligned}
&(y_0,f_0) \in \mathcal{A}, \\
&(y_{k+1},f_{k+1})=(y_k,f_k)-\lambda_k (Y^1_k,F_k^1), \quad k\in \mathbb{N}, \\
& \lambda_k= \textrm{argmin}_{\lambda\in [0,m]} E\big((y_k,f_k)-\lambda (Y^1_k,F_k^1)\big),    
\end{aligned}
\right.
\end{equation}
where $(Y^1_k,F_k^1)\in \mathcal{A}_0$ is the solution of minimal control norm of 
\begin{equation}
\label{wave-Y1k}
\left\{
\begin{aligned}
& \partial_{tt} Y_{k}^1 - \Delta Y^1_k +  g^{\prime}(y_k)\cdot Y^1_k = F^1_k 1_{\omega}+ (\partial_{tt} y_{k}-\Delta y_k+g(y_k)-f_k 1_\omega), & \textrm{in}\,\, Q_T,\\
& Y_k^1=0,  & \textrm{on}\,\, \Sigma_T, \\
& (Y_k^1(\cdot,0),\partial_{t} Y_{k}^1(\cdot,0))=(0,0), & \textrm{in}\,\, \Omega.
\end{aligned}
\right.
\end{equation}
The real number $m\geq 1$ is arbitrarily fixed and is introduced in order to keep the sequence $(\lambda_k)_{k\in\mathbb{N}}$ bounded.

Given any $s\in [0,1]$, we set
\begin{equation}\label{def_betahat}
\beta^\star(s) := \sqrt{\frac{s}{2C(2s+1)}}
\end{equation}
where $C>0$, only depending on $\Omega$ and $T$, is the constant appearing in Proposition \ref{controllability_result}. In this section, we prove our main result. 

\begin{theorem}\label{main_theorem}
Assume that $g^\prime$ satisfies \ref{constraint_g_holder} for some $s\in [0,1]$ and 
\begin{enumerate}[label=$\bf (H_2)$,leftmargin=1.5cm]
\item\label{growth_condition}\ There exists $\alpha\geq 0$ and $\beta \in [0,\beta^\star(s))$ such that $\vert g^{\prime}(r)\vert\leq \alpha + \beta \ln^{1/2}(1+\vert r\vert )$ for every $r$ in $\R$ % \label{asymptotic_behavior}
\end{enumerate}
if $s\in (0,1]$ 
and 
\begin{enumerate}[label=$\bf (H_3)$,leftmargin=1.5cm]
\item\label{bornegprime}\ $\sqrt{2}C\Vert g^{\prime}\Vert_{\infty}e^{C\Vert g^{\prime}\Vert^2_\infty \vert \Omega\vert^{2/d}}<1$
\end{enumerate}
if $s=0$.

Then, for any $(y_0,f_0)\in \A$, the sequence $(y_k,f_k)_{k\in \mathbb{N}}$ defined by \eqref{algo_LS_Y} strongly converges to a pair $(\overline{y},\overline{f})\in\A$ satisfying \eqref{eq:wave-NL} and the condition $(y(\cdot,T),y_t(\cdot,T))=(z_0,z_1)$, for all $(u_0,u_1), (z_0,z_1)\in \boldsymbol{V}$. Moreover, the convergence is at least linear and is at least of order $1+s$ after a finite number of iterations.\footnote{We recall that a sequence $(u_k)_{k\in\N}$ of real numbers converges to $0$ with order $\alpha\geq 1$ if there exists $M>0$ such that $\vert u_{k+1}\vert \leq M \vert u_k\vert^\alpha$ for every $k\in\N$. A sequence $(v_k)_{k\in\N}$ of real numbers converges to $0$ at least with order $\alpha\geq 1$ if there exists a sequence $(u_k)_{k\in\N}$ of nonnegative real numbers converging to $0$ with order $\alpha\geq 1$ such that $\vert v_k\vert\leq u_k$ for every $k\in\N$.}

\end{theorem}

Consequently, the algorithm \eqref{algo_LS_Y} provides a constructive way to approximate a control for the nonlinear wave equation \eqref{algo_LS_Y}. 

The proof consists in showing that the decreasing sequence $(E(y_k,f_k))_{k\in \mathbb{N}}$ converges to zero. In view of \eqref{ineq_E_Eprime}, this property is related to the uniform property of the  observability constant $e^{C\Vert g^{\prime}(y_k) \Vert^2_{L^\infty(0,T;L^d(\Omega))}}$ with respect to $k$. 
In order to fix some notations and the main ideas of the proof of Theorem \ref{main_theorem}, we first prove in Section \ref{gprimebound} the convergence of the sequence $(y_k,f_k)_{k\in \mathbb{N}}$ under the stronger condition that $g^\prime\in L^\infty(\mathbb{R})$, sufficient to ensure the boundedness of the sequence $\big(e^{C\Vert g^{\prime}(y_k) \Vert^2_{L^\infty(0,T;L^d(\Omega))}}\big)_{k\in \mathbb{N}}$. Then, in Section \ref{sec:main_theorem}, we prove Theorem \ref{main_theorem} by showing that under the assumption \ref{growth_condition}, the sequence $(y_k,f_k)_{k\in \mathbb{N}}$ is actually bounded in $\A$. This implies the same property for the real sequence $e^{C\Vert g^{\prime}(y_k) \Vert^2_{L^\infty(0,T;L^d(\Omega))}}$, and then the announced convergence.

\subsection{Proof of the convergence under the additionnal assumption $g^\prime\in L^\infty(\mathbb{R})$}\label{gprimebound}

We establish in this section the following preliminary result, which coincides with Theorem \ref{main_theorem} in the simpler case $\beta=0$.

 \begin{prop}\label{prop_gprime_bounded}
Assume that $g^\prime$ satisfies \ref{constraint_g_holder} for some $s\in [0,1]$ and that $g^\prime\in L^\infty(\mathbb{R})$. If $s=0$, assume moreover \ref{bornegprime}.  For any $(y_0,f_0)\in \A$, the sequence $(y_k,f_k)_{k\in\mathbb{N}}$ defined by (\ref{algo_LS_Y}) strongly converges to a pair $(\overline{y},\overline{f})\in\A$ satisfying \eqref{eq:wave-NL} and the condition $(y(\cdot,T),y_t(\cdot,T))=(z_0,z_1)$, for all $(u_0,u_1), (z_0,z_1)\in \boldsymbol{V}$. Moreover, the convergence is at least linear and is at least of order $1+s$ after a finite number of iterations.
\end{prop}

Proceeding as in \cite{lemoinemunch_NUMER,munch_trelat}, Proposition \ref{prop_gprime_bounded} follows from the following lemma. 
  
\begin{lemma}\label{geomquadratic}
Under the hypotheses of Proposition \ref{prop_gprime_bounded}, for any $(y_0,f_0)\in \A$, there exists a $k_0\in \mathbb{N}$ such that the sequence  $(E(y_k,f_k))_{k\geq k_0}$ tends to $0$ as $k\to\infty$ with at least a rate $s+1$. 
\end{lemma}
\begin{proof} Since $g^{\prime}\in L^\infty(\mathbb{R})$, the nonnegative constant
$c(y_k)$ in \eqref{estimateEs} is uniformly bounded w.r.t. $k$: we introduce the real $c>0$ 
as follows 
\begin{equation}\label{def_c}
c(y_k)\leq c:=\frac{C}{(1+s)\sqrt{2}} [g^\prime]_s  \biggl(C  e^{C\Vert g^{\prime}\Vert^2_{\infty} \vert \Omega\vert^{2/d}}\biggr)^{1+s}, \quad \forall k\in \mathbb{N}.
\end{equation}
$\vert \Omega\vert$ denotes the measure of the domain $\Omega$.
 For any $(y_k,f_k)\in \mathcal{A}$, let us then denote the real function $p_k$ by 
$$
p_k(\lambda):=\vert 1-\lambda\vert +\lambda^{1+s} c\, E(y_k,f_k)^{s/2}, \quad \lambda\in [0,m].
$$
Lemma \ref{lemma_estimateEs} with $(y,f)=(y_k,f_k)$ then allows to write that 
 \begin{equation}\label{estek}
\sqrt{E(y_{k+1},f_{k+1})}=\min_{\lambda\in [0,m]} \sqrt{E((y_k,f_k)-\lambda (Y^1_k,F^1_k))}\leq p_k(\widetilde{\lambda_k}) \sqrt{E(y_k,f_k)}
\end{equation}
with $p_k(\widetilde{\lambda_k}):=\min_{\lambda\in [0,m]} p_k(\lambda)$.
Assume first that $s>0$. We then easily check that the optimal $\widetilde{\lambda_k}$ is given by 
\begin{equation}
\widetilde{\lambda_k}:=\left\{
\begin{aligned}
& \frac{1}{(1+s)^{1/s} c^{1/s}\sqrt{E(y_k,f_k)}}, & \textrm{if}\quad (1+s)^{1/s} c^{1/s}\sqrt{E(y_k,f_k)}\geq 1,\\
& 1, & \textrm{if}\quad (1+s)^{1/s} c^{1/s}\sqrt{E(y_k,f_k)}< 1
\end{aligned}
\right.
\end{equation}
leading to 
\begin{equation} \label{ptildek}
p_k(\widetilde{\lambda_k}):=\left\{
\begin{aligned}
& 1- \frac{s}{(1+s)^{\frac1s+1}}\frac{1}{c^{1/s}\sqrt{E(y_k,f_k)}}, & \textrm{if}\quad (1+s)^{1/s} c^{1/s}\sqrt{E(y_k,f_k)}\geq 1,\\
& c\, E(y_k,f_k)^{s/2}, & \textrm{if}\quad (1+s)^{1/s} c^{1/s}\sqrt{E(y_k,f_k)}< 1.
\end{aligned}
\right.
\end{equation}
Accordingly, we may distinguish two cases : 
\vspace{0.25cm}
\par\noindent
$\bullet$ If $(1+s)^{1/s} c^{1/s}\sqrt{E(y_0,f_0)}< 1$, then  $c^{1/s}\sqrt{E(y_0,f_0)}< 1$, and thus $c^{1/s}\sqrt{E(y_k,f_k)}<1$ for all $k\in\mathbb{N}$ since the sequence $(E(y_k,f_k))_{k\in \mathbb{N}}$ is decreasing. Hence \eqref{estek} implies that 
$$
c^{1/s}\sqrt{E(y_{k+1},f_{k+1})}\le \big(c^{1/s}\sqrt{E(y_k,f_k)}\big)^{1+s} \quad \forall k\in \mathbb{N}.
$$
It follows that $c^{1/s}\sqrt{E(y_k,f_k)}\to 0$ as $k\to \infty$ with a rate equal to $1+s$.
\vspace{0.25cm}
\par\noindent
$\bullet$ If $(1+s)^{1/s} c^{1/s}\sqrt{E(y_0,f_0)}\geq 1$ then we check that the set $I:=\{k\in \mathbb{N},\ (1+s)^{1/s} c^{1/s}\sqrt{E(y_k,f_k)}\geq 1\}$ is a finite subset of $\mathbb{N}$;  indeed, for all $k\in I$, \eqref{estek} implies that 
\begin{equation}\label{decayEunquarts}
c^{1/s}\sqrt{E(y_{k+1},f_{k+1}) } \le \Big(1- \frac{s}{(1+s)^{\frac1s+1}}\frac{1}{c^{1/s}\sqrt{E(y_k,f_k)}}\Big)c^{1/s}\sqrt{E(y_k,f_k)}=c^{1/s}\sqrt{E(y_k,f_k) }-\frac{s}{(1+s)^{\frac1s+1}}
\end{equation}
and the strict decrease of the sequence $(c^{1/s} \sqrt{E(y_k,f_k)})_{k\in I}$. Thus there exists 
$k_0\in\mathbb{N}$ such that for all $k\geq k_0$, $(1+s)^{1/s} c^{1/s}\sqrt{E(y_k,f_k)}<1$, that is  $I$ is a finite subset of $\mathbb{N}$.  
Arguing as in the first case, it follows that $\sqrt{E(y_k,f_k)}\to 0$ as $k\to \infty$. 

It follows in particular from (\ref{ptildek}) that the sequence $(p_k(\widetilde{\lambda_k}))_{k\in\N}$ decreases as well.

If now $s=0$, then $p_k(\lambda)=\vert 1-\lambda\vert + \lambda c$ with $c=[g^\prime]_0  C  e^{C\Vert g^{\prime}\Vert^2_{\infty} \vert \Omega\vert^{2/d}}$ and \eqref{estek} with $\widetilde{\lambda_k}=1$ leads to $\sqrt{E(y_{k+1},f_{k+1})}\leq c \sqrt{E(y_k,f_k)}$. The convergence of $(E(y_k,f_k))_{k\in \mathbb{N}}$ to $0$ holds if $c<1$, i.e. \ref{bornegprime}.

\end{proof}
\begin{proof}(of Proposition \ref{prop_gprime_bounded}) In view of \eqref{estimateF1Y1_A0}, we write 
$$
 \big(1+\Vert g^\prime(y)\Vert_{L^\infty(0,T;L^3(\Omega))}\big)e^{C \Vert g^{\prime}(y)\Vert^2_{\infty(0,T;L^d(\Omega))} }\leq (1+\Vert g^{\prime}\Vert_{\infty}\vert \Omega\vert^{1/3})e^{C\Vert g^{\prime}\Vert^2_{\infty} \vert \Omega\vert^{2/d}}\leq e^{2C\Vert g^{\prime}\Vert^2_{\infty} \vert \Omega\vert^{2/d}}
$$
using that $(1+u)e^{u^2}\leq e^{2u^2}$ for all $u\in \R^+$. It follows that 

\begin{equation}\label{estim_sum_k}
\sum_{n=0}^k \vert \lambda_n \vert \Vert (Y^1_n,F^1_n)\Vert_{\mathcal{H}} \leq m\, C e^{C\Vert g^{\prime}\Vert^2_{\infty} \vert \Omega\vert^{2/d}} \sum_{n=0}^k  \sqrt{E(y_n,f_n)}.
\end{equation}
Using that  
$p_n(\widetilde{\lambda}_n)\leq p_0(\widetilde{\lambda}_0)$ for all $n\geq 0$, we can write for $n>0$,
\begin{equation}\label{ineqEn}
\sqrt{E(y_n,f_n)}\leq p_{n-1}(\widetilde{\lambda}_{n-1})\sqrt{E(y_{n-1},f_{n-1})}\leq p_{0}(\widetilde{\lambda}_0)\sqrt{E(y_{n-1},f_{n-1})}\leq (p_{0}(\widetilde{\lambda}_0))^n\sqrt{E(y_0,f_0)}.
\end{equation}
Then, using that $p_0(\widetilde{\lambda}_{0})=\min_{\lambda\in [0,m]}p_0(\lambda)<1$ (since $p_0(0)=1$ and $p_0^\prime(0)<0$), 
we finally obtain the uniform estimate 
$$
\sum_{n=0}^k \vert \lambda_n \vert \Vert (Y^1_n,F^1_n)\Vert_{\mathcal{H}} \leq m\, C e^{C\Vert g^{\prime}\Vert^2_{\infty} \vert \Omega\vert^{2/d}}  \frac{\sqrt{E(y_0,f_0)}}{1-p_0(\widetilde{\lambda}_{0})}
$$
for which we deduce (since $\mathcal{H}$ is a complete space) that the serie $\sum_{n\geq 0}\lambda_n (Y^1_n,F_n^1)$ converges in $\mathcal{A}_0$. Writing from \eqref{algo_LS_Y} that $(y_{k+1},f_{k+1})=(y_0,f_0)-\sum_{n=0}^k \lambda_n (Y^1_n,F_n^1)$, we conclude that $(y_k,f_k)$ strongly converges in $\mathcal{A}$ to $(\overline{y},\overline{f}):=(y_0,f_0)+\sum_{n\geq 0}\lambda_n (Y^1_n,F_n^1)$. 

Then, using that $(Y^1_k, F^1_k)$ goes to zero as $k\to \infty$ in $\mathcal{A}_0$, we pass to the limit in \eqref{wave-Y1k} and get that $(\overline{y},\overline{f})\in \mathcal{A}$ solves 
\begin{equation}
\label{wave-limit}
\left\{
\begin{aligned}
& \partial_{tt}\overline{y} - \Delta \overline{y} +  g(\overline{y}) = \overline{f} 1_{\omega}, &\textrm{in}\quad Q_T,\\
& \overline{y}=0, & \textrm{on}\,\, \Sigma_T, \\
& (\overline{y}(\cdot,0),\partial_{t}\overline{y}(\cdot,0))=(y_0,y_1), & \textrm{in}\,\, \Omega.
\end{aligned}
\right.
\end{equation}
Since the limit $(\overline{y},\overline{f})$ belongs to $\mathcal{A}$, we have that $(\overline{y}(\cdot,T),\overline{y}_t(\cdot,T))=(z_0,z_1)$ in $\Omega$. Moreover, for all $k>0$
\begin{equation}
\label{estim_coercivity}
\begin{aligned}
\Vert (\overline{y},\overline{f})-(y_k,f_k)\Vert_{\mathcal{H}} & =\biggl\Vert \sum_{p=k+1}^{\infty} \lambda_p (Y^1_p,F^1_p)\biggr\Vert_{\mathcal{H}}\leq m\sum_{p=k+1}^{\infty}  \Vert (Y^1_p,F^1_p \Vert_{\mathcal{H}}\\ 
& \leq m\, C\sum_{p=k+1}^{\infty}  \sqrt{E(y_p,f_p)}\leq m\, C\sum_{p=k+1}^{\infty}  p_{0}(\widetilde{\lambda}_0)^{p-k}\sqrt{E(y_k,f_k)}\\
& \leq m\, C\frac{p_{0}(\widetilde{\lambda}_0)}{1-p_{0}(\widetilde{\lambda}_0)}\sqrt{E(y_{k},f_{k})}
\end{aligned}
\end{equation}
and conclude from Lemma \ref{geomquadratic} the convergence of order at least $1+s$ after a finite number of iterates. 
\end{proof}

\begin{remark}
In particular, along the sequence $(y_k,f_k)_{k\in \mathbb{N}}$ defined by (\ref{algo_LS_Y}), \eqref{estim_coercivity} is a kind of coercivity property for the functional $E$.
We emphasize, in view of the non uniqueness of the zeros of $E$, that an estimate (similar to \eqref{estim_coercivity}) of the form  $\Vert (\overline{y},\overline{f})-(y,f)\Vert_{\mathcal{H}} \leq C \sqrt{E(y,f)}$ does not hold for all $(y,f)\in\mathcal{A}$. We also insist in the fact the sequence $(y_k,f_k)_{k\in \mathbb{N}}$ and its limits $(\overline{y},\overline{f})$ are uniquely determined from the initialization $(y_0,f_0)\in \mathcal{A}$ and from our selection criterion for the control $F^1$. %In other words, the solution $(\overline{y},\overline{f})$ is unique up to the element $(y_0,f_0)\in \A$.
\end{remark}

\begin{remark}
Estimate \eqref{estim_sum_k} implies the uniform estimate on the sequence $(\Vert (y_k,f_k)\Vert_{\mathcal{H}})_{k\in \mathbb{N}}$:
$$
\begin{aligned}
\Vert (y_k,f_k)\Vert_{\mathcal{H}} &\leq \Vert (y_0,f_0)\Vert_{\mathcal{H}}+ m\, C e^{C\Vert g^{\prime}\Vert^2_{\infty} \vert \Omega\vert^{2/d}} \sum_{n=0}^{k-1}  \sqrt{E(y_n,f_n)}\\
&\leq \Vert (y_0,f_0)\Vert_{\mathcal{H}}+ m\, C e^{C\Vert g^{\prime}\Vert^2_{\infty} \vert \Omega\vert^{2/d}} \frac{\sqrt{E(y_0,f_0)}}{1-p_0(\widetilde{\lambda}_{0})}.
\end{aligned}
$$
In particular, for $s>0$ and the less favorable case for which  $(1+s)^{1/s}c^{1/s}\sqrt{E(y_0,f_0)}\geq 1$, we get 
$
\frac{\sqrt{E(y_0,f_0)}}{1-p_0(\widetilde{\lambda}_{0})}= \frac{(1+s)^{\frac{1}{s}+1}}{s}c^{1/s} E(y_0,f_0), 
$
(see  \eqref{ptildek}) leading to 
$$
\Vert (y_k,f_k)\Vert_{\mathcal{H}} \leq \Vert (y_0,f_0)\Vert_{\mathcal{H}}+ m\, C e^{C \Vert g^\prime\Vert^2_\infty \vert \Omega\vert^{2/d}} \frac{(1+s)^{\frac{1}{s}+1}}{s}c^{1/s} E(y_0,f_0), \\
$$
and then, in view of \eqref{def_c} to the explicit estimate in term of the data
\begin{equation}
%\boxed{
\Vert (y_k,f_k)\Vert_{\mathcal{H}} \leq \Vert (y_0,f_0)\Vert_{\mathcal{H}}+ m\, \frac{(1+s)}{s}\, \biggl(\frac{C[g^\prime]_s}{\sqrt{2}} \biggr)^{1/s}\, \biggl(C e^{C\Vert g^{\prime}\Vert^2_{\infty} \vert \Omega\vert^{2/d}}\biggr)^{\frac{2s+1}{s}} E(y_0,f_0).
%}
 \end{equation}

The case $s=0$ under the hypothesis $c<1$ leads to $\Vert (y_k,f_k)\Vert_{\mathcal{H}}\leq \Vert (y_0,f_0)\Vert_{\mathcal{H}}+ m\, \frac{c\sqrt{E(y_0,f_0)}}{1-c}$.

\end{remark}

\begin{remark}
For $s>0$, recalling that the constant $c$ is defined in \eqref{def_c},
if $(1+s)^{1/s}c^{1/s}\sqrt{E(y_0,f_0)}\geq 1$, inequality \eqref{decayEunquarts} implies that 

\begin{equation}
c^{1/s}\sqrt{E(y_k,f_k) } \le c^{1/s}\sqrt{E(y_0,f_0) }-k\frac{s}{(1+s)^{\frac1s+1}}, \quad \forall k\in I.
\end{equation}
Hence, the number of iteration $k_0$ to achieve a rate $1+s$ is estimate as follows : 
$$
k_0=\biggl\lfloor  (1+s)\biggl( c^{1/s}(1+s)^{1/s}\sqrt{E(y_0,f_0)}\biggr)-\frac{1}{s}\biggr\rfloor+1
$$
where $\lfloor \cdot\rfloor$ denotes the integer part. As expected, this number increases with $\sqrt{E(y_0,f_0)}$ and $\Vert g^{\prime}\Vert_\infty$. If $(1+s)^{1/s}c^{1/s}\sqrt{E(y_0,f_0)}<1$, then $k_0=0$. In particular, as $s\to 0^+$, $k_0\to \infty$ if $c>1$, i.e. if \ref{bornegprime} does not hold.

For $s=0$, the inequality $\sqrt{E(y_{k+1},f_{k+1})}\leq c \sqrt{E(y_k,f_k)}$ with $c<1$ leads to $k_0=0$.
\end{remark}

We also have the following convergence result for the optimal sequence $(\lambda_k)_{k>0}$. 

\begin{lemma}\label{lambda-k-go-1}
Assume that $g^\prime$ satisfies \ref{constraint_g_holder} for some $s\in [0,1]$ and that $g^\prime\in L^\infty(\mathbb{R})$. The sequence $(\lambda_k)_{k>k_0}$ defined in (\ref{algo_LS_Y}) converges to $1$ as $k\to \infty$ at least with order $1+s$.
\end{lemma}
\begin{proof} In view of \eqref{expansionE}, we have, as long as  $E(y_k,f_k)>0$, since $\lambda_k\in[0,m]$ 
$$\begin{aligned}
(1-\lambda_k)^2
&=\frac{E(y_{k+1},f_{k+1})}{E(y_{k},f_k)}-2(1-\lambda_k)\frac{\langle \big(y_{k,tt}+\Delta y_k+g(y_k)-f_k\,1_\omega\big),l(y_k,\lambda_k Y_k^1) \rangle_{2}}{E(y_{k},f_k)}\\
& \hspace{3cm}- \frac{\bigl \Vert l(y_k,\lambda_k Y_k^1)\bigr\Vert^2_{2}}{2E(y_{k})}\\
&\le \frac{E(y_{k+1},f_{k+1})}{E(y_{k},f_k)}-2(1-\lambda_k)\frac{\langle \big(y_{k,tt}+\Delta y_k+g(y_k)-f_k\,1_\omega\big),l(y_k,\lambda_k Y_k^1) \rangle_{2}}{E(y_{k},f_k)}\\
&\le \frac{E(y_{k+1},f_{k+1})}{E(y_{k},f_k)}+2\sqrt{2}m\frac{\sqrt{E(y_k,f_k)}\|l(y_k,\lambda_k Y_k^1) \|_{L^2(Q_T)}}{E(y_{k},f_k)}\\
&\le \frac{E(y_{k+1},f_{k+1})}{E(y_{k},f_k)}+2\sqrt{2}m\frac{\|\l(y_k,\lambda_k Y_k^1) \|_{2}}{\sqrt{E(y_k,f_k)}}.
\end{aligned}
$$
But, from \eqref{majoration-l} and  \eqref{estim1}, we have  
$
\|l(y_k,\lambda_k Y_k^1) \|_{L^2(Q_T)}
\le c \lambda_k^{1+s} E(y_k,f_k)^{\frac{1+s}{2}}\leq c m^{1+s}E(y_k,f_k)^{\frac{1+s}{2}}
$
and thus 
$$
(1-\lambda_k)^2\le \frac{E(y_{k+1},f_{k+1})}{E(y_{k},f_k)}+2\sqrt{2}m^{2+s} c E(y_k,f_k)^{s/2}.
$$
Consequently, since $E(y_{k},f_k)\to 0$ and $\frac{E(y_{k+1},f_{k+1})}{E(y_{k},f_k)}\to 0$ at least with order $1+s$, we deduce the result. 
\end{proof}

%As it is possibly more natural, we can directly show the geometric decay of the sequence $(\Vert y_{k+1}-y_{k}\Vert_2)_{k\in \mathbb{N}}$ (from which we deduce the convergence of $(y_k)_{k\in \mathbb{N}}$). To do so, we remark that the pair $(y_{k+1},f_{k+1})=(y_k,f_k)-\lambda_k (Y^1_k,F^1_k)$ defined by \eqref{algo_LS_Y} satisfies the equality 
%$$
%\begin{aligned}
%y_{k+1,tt}-\Delta y_{k+1}+g(y_{k+1})-f_{k+1}\, 1_\omega  =(1-\lambda_k)&\biggl(y_{k,tt}-\Delta y_k+g(y_k)-f_k\, 1_\omega\biggr)\\
%& + g(y_k-\lambda_k Y_k^1)-g(y_k)+ \lambda_k g^{\prime}(y_k)Y_k^1 
%\end{aligned}
%$$
%where $(Y_k^1,F_k^1)$ is a null controlled pair and $\lambda_k>0$ the parameter which minimize the $L^2$ norm of the right hand side. Estimate \eqref{estimateF1Y1} leads to
%%
%\begin{equation} %\label{estimateY1}
%\nonumber
%\begin{aligned}
%\Vert (Y_k^1,Y_{k,t}^1)\Vert_{L^{\infty}(0,T;\boldsymbol{V})}& \leq C_1  e^{C_2\sqrt{\Vert g^{\prime}(y_k)\Vert_{\infty}} }\sqrt{E(y_k,f_k)}\\
%& \leq C_1  e^{C_2\sqrt{\Vert g^{\prime}(y_k)\Vert_{\infty}} }\biggl((1-\lambda_{k-1})\sqrt{E(y_{k-1},f_{k-1})}+ \Vert g(y_k-\lambda_k Y_k^1)-g(y_k)+ \lambda_k g^{\prime}(y_k)Y_k^1 \Vert_2 \biggr) \\
%     & \leq C_1  e^{C_2\sqrt{\Vert g^{\prime}(y_k)\Vert_{\infty}} }
% \biggl((1-\lambda_{k-1})\sqrt{E(y_{k-1},f_{k-1})}+ \frac{\lambda_{k-1}^2}{2}\Vert g^{\prime\prime}\Vert_\infty\Vert (Y_{k-1}^1)^2 \Vert_2 \biggr).
%\end{aligned}
%\end{equation}
%The geometric (then quadratic) decay of $\Vert Y_k\Vert_2$ (i.e. $\Vert y_{k+1}-y_k\Vert_2$) then follows arguing as in the proof of Proposition \ref{geomquadratic}.

\subsection{Proof of Theorem \ref{main_theorem}}\label{sec:main_theorem}

In this section, we relax the condition $g^\prime\in L^\infty(\R)$ and prove Theorem \ref{sec:main_theorem}, for $s>0$ under the assumption \ref{growth_condition}. This assumption implies notably that $\vert g(r)\vert \leq C(1+\vert r\vert)\ln(2+\vert r\vert)$ for every $r\in \mathbb{R}$, mentioned in the introduction to state the well-posedness of \eqref{eq:wave-NL}.
The case $\beta=0$ corresponds to the case developed in the previous section, i.e. $g^\prime\in L^\infty(\mathbb{R})$.

Within this more general framework, the difficulty is to have a uniform control with respect to $k$ of the observability constant  $Ce^{C\Vert g^{\prime}(y_k)\Vert^2_{L^\infty(0,T;L^d(\Omega))}}$ appearing in the estimates for $(Y_k^1,F_k^1)$, see Proposition \ref{estimateC1C2}. In other terms, we have to show that the sequence $(y_k,f_k)_{k\in \mathbb{N}}$ uniquely defined in \eqref{algo_LS_Y} is uniformly bounded in $\A$, for any $(y_0,f_0)\in \A$.

We need the following intermediate result.

\begin{lemma}\label{estimate_Cobs} Let $C>0$, only depending on $\Omega$ and $T$ be the constant appearing in Proposition \ref{controllability_result}. Assume that $g$ satisfies the growth condition \ref{growth_condition} and $2C\beta^2\leq 1$.
Then for any $(y,f)\in \A$,
$$
e^{C\Vert g^\prime(y)\Vert^2_{L^\infty(0,T;L^d(\Omega))}}\leq 2\, C\max (1,e^{2C\alpha^2} \vert \Omega\vert^2) \biggl(1+\frac{\Vert y\Vert_{L^\infty(0,T;L^{p^\star}(\Omega))}}{\vert \Omega\vert^{1/p^\star}} \biggr)^{2 C\beta^2}
$$
for any  $p^\star\in \mathbb{N}^\star$ with $p^\star<\infty$ if $d=2$ and $p^\star\leq 6$ if $d=3$.
\end{lemma}

\begin{proof} We use the following inequality (direct consequence of the inequality (3.8) in \cite{Li_Zhang_2000}):
\begin{equation}\label{estim_zhang}
e^{C\Vert g^\prime(y)\Vert^2_{L^\infty(0,T;L^d(\Omega))}}\leq C\biggl(1+\sup_{t\in (0,T)}\int_{\Omega} e^{C\vert g^\prime(y)\vert^2}\biggr), \quad \forall (y,f)\in \A.
\end{equation}

Writing that $\vert g^\prime(y)\vert^2 \leq 2 \big(\alpha^2+\beta^2 \ln(1+\vert y\vert)\big)$, we get that $\int_{\Omega} e^{C\vert g^\prime(y)\vert^2} \leq e^{2C\alpha^2} \int_{\Omega} (1+\vert y\vert)^{2C\beta^2}$. Assuming $2C\beta^2\leq p^\star$, Holder inequality leads to 
$$
\begin{aligned}
\int_{\Omega} e^{C\vert g^\prime(y)\vert} & \leq e^{2C\alpha^2} \biggl(\int_{\Omega} (1+\vert y\vert)^{p^\star}\biggr)^{\frac{2C\beta^2}{p^\star}} \vert \Omega\vert^{1-\frac{2C\beta^2}{p^\star}}\\
& \leq e^{2C\alpha^2} \vert \Omega\vert  \biggl(1+\frac{\Vert y\Vert_{L^{p\star}(\Omega)}}{\vert \Omega\vert^{1/p^\star}}\biggr)^{2C\beta^2}.
\end{aligned}
$$
It follows, by \eqref{estim_zhang}, that for every $(y,f)\in \A$, 

$$
\begin{aligned}
e^{C\Vert g^\prime(y)\Vert^2_{L^\infty(0,T;L^d(\Omega))}} & \leq C\biggl(1+e^{2C\alpha^2} \vert \Omega\vert\biggl(1+\frac{\Vert y\Vert_{L^\infty(0,T;L^{p^\star}(\Omega))}}{\vert \Omega\vert^{1/p^\star}} \biggr)^{2C\beta^2}\biggr) \\
& \leq C \max (1,e^{2C\alpha^2} \vert \Omega\vert)\biggl(1+ \biggl(1+\frac{\Vert y\Vert_{L^\infty(0,T;L^{p^\star}(\Omega))}}{\vert \Omega\vert^{1/p^\star}} \biggr)^{2C\beta^2}\biggr) \\
& \leq 2^{2C\beta^2}\, C \max (1,e^{2C\alpha^2} \vert \Omega\vert)\biggl(1+\frac{\Vert y\Vert_{L^\infty(0,T;L^{p^\star}(\Omega))}}{\vert \Omega\vert^{1/p^\star}} \biggr)^{2C\beta^2}
\end{aligned}
$$
and the result.
\end{proof}

%
%\
%
%In the sequel, we define the pair $(y^\star,f^\star)\in \A$ such that $f^\star$ is the control of minimal $L^2(q_T)$-norm for $y^\star$ solution of  \eqref{eq:wave-NL} with $g\equiv 0$.
%
%
%\begin{lemma}\label{estimate_Ey0f0}
%Assume that $g$ satisfies \eqref{asymptotic_behavior_gprime_concrete}. Then, for all $\beta_0>0$, there exists a constant $C(\beta_0)$ such that $\sup_{\beta\in [0,\beta_0]} E(y^\star,f^\star)<C(\beta_0)$.
%\end{lemma}
%\textsc{Proof}-We get that $E(y^\star,f^\star)=\frac{1}{2}\Vert y_{tt}^\star-\Delta y^\star +g(y^\star)-f^\star \,1_{\omega}\Vert_2^2=\frac{1}{2}\Vert g(y^\star)\Vert_2^2$. In view of \eqref{asymptotic_behavior_gprime_concrete},
%it follows that 
%$$
%\sqrt{E(y^\star,f^\star)}\leq  T \vert g(0)\vert + T\biggl(\alpha +\beta \log^{1/2}(1+\Vert y^\star\Vert_{L^\infty(0,T;H_0^1(\Omega))})\biggr)\Vert y^\star\Vert_{L^\infty(0,T;H_0^1(\Omega))}
%$$
%Since $y^\star$ is independent of $g$ and therefore of $\beta$, the result follows.
%$\hfill\Box$
%

\begin{lemma}\label{lemma_jensen}
Assume that $g$ satisfies the growth condition \ref{growth_condition} and $2C\beta^2\leq 1$.
 For any $(y,f)\in \A$, the unique solution $(Y^1,F^1)\in A_0$ of \eqref{wave-Y1} satisfies
\begin{equation} 
\Vert (Y^1,Y_t^1)\Vert_{L^{\infty}(0,T;\boldsymbol{V})}+ \Vert F^1\Vert_{2,q_T} \leq d(y)\sqrt{E(y,f)}\end{equation}
with 
\begin{equation}
d(y):=C_3(\alpha)\biggl(1+\frac{\Vert y\Vert_{L^\infty(0,T;L^1(\Omega))}}{\vert \Omega\vert} \biggr)^{2 C \beta^2}, \quad C_3(\alpha):=2\, C \max (1,e^{2C\alpha^2} \vert \Omega\vert).
\end{equation}
\end{lemma}
\begin{proof}
Lemma \ref{estimate_Cobs}  with $p^\star=1$ and \eqref{estimateF1Y1}  lead to the result.
\end{proof}

With these notations, the term $c(y)$ in \eqref{estimateEs} rewrites as 
\begin{equation}\label{defc(y)}
c(y)=\frac{C}{(1+s)\sqrt{2}} [g^{\prime}]_s \, d(y)^{1+s}, \quad \forall (y,f)\in \A, \,\forall s\in (0,1].
\end{equation}
\begin{proof}(of Theorem \ref{main_theorem})
If the initialization $(y_0,f_0)\in \A$ is such that $E(y_0,f_0)=0$, then the sequence $(y_k,f_k)_{k\in \mathbb{N}}$ is constant equal to $(y_0,f_0)$ and therefore converges. We assume in the sequel that $E(y_0,f_0)>0$.

We are going to prove that, for any $\beta<\beta^\star(s)$, there exists a constant $M>0$ such that the sequence $(y_k)_{k\in \mathbb{N}}$ defined by \eqref{algo_LS_Y} enjoys the uniform property

\begin{equation}\label{uniform_property}
\Vert y_k\Vert_{L^\infty(0,T;L^1(\Omega))}\leq M, \quad \forall k\in \mathbb{N}.
\end{equation}
The convergence of the sequence $(y_k,f_k)_{k\in \mathbb{N}}$ in $\mathcal{A}$ will then follow by proceeding as in Section \ref{gprimebound}.
Remark preliminary that the assumption $\beta<\beta^\star(s)$ implies $2C\beta^2< \frac{s}{2s+1}\leq 1$ since $s\in (0,1]$.

\vskip 0.25cm
\par\noindent
\textit{Proof of the uniform property \eqref{uniform_property} for some $M$ large enough}-
As for $n=0$, from any initialization $(y_0,f_0)$ chosen in $\A$, it suffices to take $M$ larger than $M_1:=\Vert y_0\Vert_{L^\infty(0,T;L^1(\Omega))}$. We then proceed by induction and assume that, for some $n\in\mathbb{N}$, $\Vert y_k\Vert_{L^\infty(0,T;L^1(\Omega))}\leq M$ for all $k\leq n$. 
This implies in particular that,
$$
d(y_k)\leq d_M(\beta):=C_3(\alpha)\biggl(1+\frac{M}{\vert \Omega\vert} \biggr)^{2 C\beta^2}, \quad \forall k\leq n
$$ 
and then
\begin{equation}\label{def_CMbeta}
c(y_k)\leq c_M(\beta):=\frac{C}{(1+s)\sqrt{2}}[g^\prime]_s\,  d_M^{1+s}(\beta), \quad \forall k\leq n.
\end{equation}
Then, we write that $\Vert y_{n+1}\Vert_{L^\infty(0,T;L^1(\Omega))}\leq \Vert y_0\Vert_{L^\infty(0,T;L^1(\Omega))}+\sum_{k=0}^n \lambda_k\Vert Y_k^1\Vert_{L^\infty(0,T;L^1(\Omega))}$.
But, Lemma \ref{lemma_jensen} implies that $
\Vert Y_k^1\Vert_{L^\infty(0,T;L^1(\Omega))}\leq d_M(\beta) \sqrt{E(y_k,f_k)}$ for all $k\leq n$ leading to 
\begin{equation}\label{estim_intermediaire}
\Vert y_{n+1}\Vert_{L^\infty(0,T;L^1(\Omega))}\leq \Vert y_0\Vert_{L^\infty(0,T;L^1(\Omega))}+m\, d_M(\beta) \sum_{k=0}^n \sqrt{E(y_k,f_k)}.
\end{equation}
Moreover, inequality \eqref{ineqEn} implies that $\sum_{k=0}^n \sqrt{E(y_k,f_k)}\leq \frac{1}{1-p_0(\widetilde{\lambda_0})}\sqrt{E(y_0,f_0)}$
where $p_0(\widetilde{\lambda_0})$ is given by \eqref{ptildek} with $c=c_M(\beta)$.

Now, we take  $M$ large enough so that $(1+s)^{1/s} c^{1/s}_M(\beta)\sqrt{E(y_0,f_0)}\geq 1$ 
%i.e.
%\begin{equation}\label{constraint1}
%\biggl(\frac{C}{\sqrt{2}} \Vert g^{\prime}\Vert_{\widetilde{W}^{s,\infty}(\mathbb{R})}\biggr)^{1/s} d_M^{2/s}(\beta)\geq 1
%\end{equation}
%
i.e.
\begin{equation}\label{constraint1}
\biggl(\frac{C}{\sqrt{2}} \,[g^\prime]_s\biggr)^{1/s} C_3(\alpha)^{2/s}\biggl(1+\frac{M}{\vert \Omega\vert} \biggr)^{\frac{4 C\beta^2}{s}}\sqrt{E(y_0,f_0)}\geq 1.
\end{equation}

Such $M$ exists since $\sqrt{E(y_0,f_0)}>0$ is independent of $M$ and since the left hand side is of order 
$\mathcal{O}(M^{\frac{4C\beta^2}{s}})$ with $\frac{4C\beta^2}{s}>0$. We denote by $M_2$ the smallest value of $M$
such that \eqref{constraint1} hold true. 

\par\noindent
Then, from  \eqref{ptildek}, we get that $p_0(\widetilde{\lambda_0})=1- \frac{s}{(1+s)^{\frac1s+1}}\frac{1}{c^{1/s}_M(\beta)\sqrt{E(y_0,f_0)}}$ and therefore 
$$
\frac{1}{1-p_0(\widetilde{\lambda_0})}= \frac{(1+s)^{\frac1s+1}}{s}c^{1/s}_M(\beta)\sqrt{E(y_0,f_0)}
$$
so that $\sum_{k=0}^n \sqrt{E(y_k,f_k)}\leq \frac{(1+s)^{\frac1s+1}}{s}c^{1/s}_M(\beta) E(y_0,f_0)$. It follows from \eqref{estim_intermediaire} that  
$$
\Vert y_{n+1}\Vert_{L^\infty(0,T;L^1(\Omega))}\leq \Vert y_0\Vert_{L^\infty(0,T;L^1(\Omega))}+m\, d_M(\beta) \frac{(1+s)^{\frac1s+1}}{s}c_M^{1/s}(\beta) E(y_0,f_0).
$$
The definition of $c_M(\beta)$ (see \eqref{def_CMbeta}) then gives
$$
\begin{aligned}
\Vert y_{n+1}\Vert_{L^\infty(0,T;L^1(\Omega))}\leq &\Vert y_0\Vert_{L^\infty(0,T;L^1(\Omega))}\\
&+\frac{m(1+s)}{s}  \biggl(\frac{C[g^\prime]_s}{\sqrt{2}}\biggr)^{1/s} \biggl(C_3(\alpha)\biggr)^{1+\frac{2}{s}}\, E(y_0,f_0) \biggl(1+\frac{M}{\vert \Omega\vert} \biggr)^{\frac{(2 C\beta^2)(2s+1)}{s}}.
\end{aligned}
$$
Now, we take $M>0$ large enough so that the right hand side is bounded by $M$, i.e.
\begin{equation}
\label{constraint2}
\begin{aligned}
\Vert y_0\Vert_{L^\infty(0,T;L^1(\Omega))}+\frac{m(1+s)}{s}  \biggl(\frac{C [g^\prime]_s}{\sqrt{2}} \biggr)^{1/s} \biggl(C_3(\alpha)\biggr)^{1+\frac{2}{s}}\, E(y_0,f_0) \biggl(1+\frac{M}{\vert \Omega\vert} \biggr)^{\frac{(2 C\beta^2)(2s+1)}{s}}\leq M.
\end{aligned}
\end{equation}
Such $M$ exists under the assumption $\beta<\beta^\star(s)$ i.e. $\frac{(2 C\beta^2)(2s+1)}{s}<1$.
We denote by $M_3$ the smallest value of $M$ such that \eqref{constraint2} holds true. Eventually, taking $M:=\max(M_1,M_2,M_3)$, we get 
that $\Vert y_{n+1}\Vert_{L^\infty(0,T;L^1(\Omega))}\leq M$ as well. We have then proved by induction the uniform property
\eqref{uniform_property} for some $M$ large enough. 

\vskip 0.25cm
\par\noindent
\textit{Proof of the convergence of the sequence $(y_k,f_k)_{k\in \mathbb{N}}$}- In view of Lemma \ref{estimate_Cobs} with $p^\star=1$, the uniform property \eqref{uniform_property} implies that the observability constant $C e^{C\Vert g^{\prime}(y_k)\Vert^2_{L^\infty(0,T;L^d(\Omega))}}$ appearing in the estimates for $(Y_k^1,F_k^1)$ (see Proposition \ref{estimateC1C2})
is uniformly bounded with respect to the parameter $k$. As a consequence, the constant $c(y_k)$ appearing in the instrumental estimate \eqref{estimateEs} is bounded by $c_M(\beta)$ given by \eqref{def_CMbeta}. Consequently, the developments of Section \ref{gprimebound} apply with $c=c_M(\beta)$. Theorem \ref{main_theorem} then follows from the proof of Proposition \ref{prop_gprime_bounded}. 
\end{proof}

\begin{remark}
Remark that  $M:=\max(M_2,M_3)$ since $M_3\geq M_1$. The constant $M_2$ can be made explicit since the constraint (\ref{constraint1}) implies that 
$$
\biggl(\frac{C[g^\prime]_s}{\sqrt{2}} \biggr)^{1/s} C_3(\alpha)^{2/s}\biggl(1+\frac{M}{\vert \Omega\vert} \biggr)^{\frac{4 C\beta^2}{s}}\sqrt{E(y_0,f_0)}\geq 1.
$$
i.e.
$$
 \biggl(1+\frac{M}{\vert \Omega\vert} \biggr)^{2 C\beta^2}\geq C_3(\alpha)^{-1}\sqrt{E(y_0,f_0)}^{-s/2}\biggl(\frac{C}{\sqrt{2}} [g^\prime]_s\biggr)^{-1/2}.
$$
In particular, $M_2$ is large for small values of $\sqrt{E(y_0,f_0)}$, for any $s>0$. On the other hand, the constant $M_3$ is implicit, hence whether $M_2>M_3$ or $M_3>M_2$ depend on the values of $\sqrt{E(y_0,f_0)}$ and 
$\Vert y_0\Vert_{L^\infty(0,T;L^1(\Omega))}$. Remark that $\sqrt{E(y_0,f_0)}$ can be large and $\Vert y_0\Vert_{L^\infty(0,T;L^1(\Omega))}$ small, and vice versa. 
\end{remark}

\section{Conclusion and further comments} \label{sec:remarks}
Exact controllability of \eqref{eq:wave-NL} has been established in \cite{Zhang2007}, under a growth condition on $g$, by means of a Leray-Schauder fixed point argument that is not constructive. In this paper, under a slightly stronger growth condition and under the additional assumption that $g^\prime$ is uniformly H\"older continuous with exponent $s\in[0,1]$, we have designed an explicit algorithm and proved its convergence of a controlled solution of \eqref{eq:wave-NL}. Moreover, the convergence is super-linear of order greater than or equal to $1+s$ after a finite number of iterations. 

In turn, our approach gives a new and constructive proof of the exact controllability of \eqref{eq:wave-NL}. Moreover, we emphasize that the method is general and may be applied to any other equations or systems - not necessarily of hyperbolic nature - for which a precise observability estimate for the linearized problem is available:  we refer to \cite{lemoine_gayte_munch} addressing the case of the heat equation. Among the open issues, we mention the extension of this constructive approach to the case of the boundary controllability (see for instance \cite{zuazua91}).
\medskip

Several comments are in order.

\paragraph{Asymptotic condition.}
The asymptotic condition \ref{growth_condition} on $g^\prime$ is slightly stronger than the asymptotic condition \ref{asymptotic_behavior} made in \cite{Zhang2007}: this is due to our linearization of \eqref{eq:wave-NL} which involves $r\to g^{\prime}(r)$ while the linearization \eqref{NL_z} in \cite{Zhang2007} involves $r\to (g(r)-g(0))/r$. There exist cases covered by Theorem \ref{ZhangTH} in which exact controllability for \eqref{eq:wave-NL} is true but that are not covered by Theorem \ref{main_theorem}. Note however that the example $g(r)=a+b r+ c r \ln^{1/2}(1+\vert r\vert)$, for any $a,b\in \mathbb{R}$ and for any $c>0$  small enough (which is somehow the limit case in Theorem \ref{ZhangTH}) satisfies \ref{growth_condition} as well as 
\ref{constraint_g_holder} for any $s\in [0,1]$. 

While Theorem \ref{ZhangTH} was established in \cite{Zhang2007} by a nonconstructive Leray-Schauder fixed point argument, we obtain here, in turn, a new proof of the exact controllability of semilinear multi-dimensional wave equations, which is moreover constructive, with an algorithm that converges unconditionally, at least with order $1+s$.

\paragraph{Minimization functional.}
Among all possible admissible controlled pair $(y,v)\in \A_0$, we have selected the solution $(Y_1,F_1)$ of \eqref{wave-Y1} that minimizes the functional $J(v)=\Vert v\Vert^2_{2,q_T}$. This choice has led to the estimate \eqref{estimateF1Y1} which is one of the key points of the convergence analysis. The analysis remains true when one considers the quadratic functional $J(y,v)=\Vert w_1  v\Vert^2_{2,q_T} + \Vert w_2  y\Vert^2_2$ for some positive weight functions $w_1$ and $w_2$ (see for instance \cite{cindea_efc_munch_2013}). 

\paragraph{Link with Newton method.}
Defining $F:\mathcal{A}\to L^2(Q_T)$ by $F(y,f):=(\partial_{tt} y-\Delta y + g(y)-f\,1_\omega)$, we have $E(y,f)=\frac{1}{2}\Vert F(y,f)\Vert_2^2$ and we observe that, for $\lambda_k=1$, the algorithm \eqref{algo_LS_Y} coincides with the Newton algorithm associated to the mapping $F$ (see \eqref{Newton-nn}). This explains the super-linear convergence property in Theorem \ref{main_theorem}, in particular the quadratic convergence when $s=1$. The optimization of the parameter $\lambda_k$ gives to a global convergence property of the algorithm and leads to the so-called damped Newton method applied to $F$. For this method, global convergence is usually achieved with linear order under general assumptions (see for instance \cite[Theorem 8.7]{deuflhard}). As far as we know, the analysis of damped type Newton methods  for partial differential equations has deserved very few attention in the literature. We mention \cite{lemoinemunch_time, saramito} in the context of fluids mechanics.   

\paragraph{A variant.}
To simplify, let us take $\lambda_k=1$, as in the standard Newton method. Then, for each $k\in\N$, the optimal pair $(Y_k^1,F_k^1)\in \mathcal{A}_0$ is such that the element $(y_{k+1},f_{k+1})$ minimizes over $\A$ the functional $(z,v)\to J(z-y_k,v-f_k)$ with $J(z,v):=\Vert v\Vert_{2,q_T}$ (control of minimal $L^2(q_T)$ norm). 
Alternatively, we may select the pair $(Y_k^1,F_k^1)$ so that the element $(y_{k+1},f_{k+1})$ minimizes the functional $(z,v)\to J(z,v)$. This leads to the sequence $(y_k,f_k)_{k\in\N}$ defined by 
 \begin{equation}
\label{eq:wave_lambdakequal1}
\left\{
\begin{aligned}
& \partial_{tt}y_{k+1} - \Delta y_{k+1} +  g^{\prime}(y_k) y_{k+1} = f_{k+1} 1_{\omega}+g^\prime(y_k)y_k-g(y_k) & \textrm{in}\  Q_T,\\
& y_k=0,  & \textrm{on}\  \Sigma_T, \\
& (y_{k+1}(\cdot,0), \partial_t y_{k+1}(\cdot,0))=(u_0,u_1) & \textrm{in}\  \Omega.
\end{aligned}
\right.
\end{equation}
In this case, for every $k\in\N$, $(y_k,f_k)$ is a controlled pair for a linearized wave equation, while, in the case of the algorithm \eqref{algo_LS_Y}, $(y_k,f_k)$ is a sum of controlled pairs $(Y^1_j,F^1_j)$ for $0\leq j\leq k$.
This formulation used in \cite{EFC-AM-2012} is different and the convergence analysis (at least in the least-squares setting) does not seem to be straightforward because the term $g^\prime(y_k)y_k-g(y_k)$ is not easily bounded in terms of $\sqrt{E(y_k,f_k)}$.

\paragraph{Initialization with the controlled pair of the linear equation.} The number  of iterates to achieve convergence (notably to enter in a super-linear regime) depends on the size of the value $E(y_0,f_0)$. A natural example of an initialization $(y_0,f_0)\in\A$ is to take $(y_0,f_0)=(y^\star,f^\star)$, the unique solution of minimal control norm of \eqref{eq:wave-NL} with $g=0$ (i.e., in the linear case). Under the assumption \ref{growth_condition}, this leads to the estimate 
$$
E(y_0,f_0)=\frac{1}{2}\Vert g(y_0)\Vert_2^2\leq \vert g(0)\vert^2 \vert Q_T\vert + 2\int_{Q_T} \vert y_0\vert^2\big(\alpha^2+\beta^2  \ln(1+\vert y_0\vert)\big).
$$ 

\paragraph{Local controllability when removing the growth condition \ref{growth_condition}.}
If the real $E(y_0,f_0)$ is small enough, then we may remove the growth condition \ref{growth_condition} on $g^\prime$.
\begin{prop}\label{prop_smallness}
Assume $g^\prime$ satisfies \ref{constraint_g_holder} for some $s\in [0,1]$. Let $(y_k,f_k)_{k>0}$ be the sequence of $\mathcal{A}$ defined in \eqref{algo_LS_Y}. There exists a constant $C([g^\prime]_s)$ such that if $E(y_0,f_0)\leq C([g^\prime]_s)$, then $(y_k,f_k)_{k\in \mathbb{N}}\to (\overline{y},\overline{f})$ in $\mathcal{A}$ where $\overline{f}$ is a null control for $\overline{y}$ solution of  \eqref{eq:wave-NL}. Moreover, the convergence is at least linear and is at least of order $1+s$ after a finite number of iterations.
\end{prop}
\begin{proof} In this proof, the notation $\Vert \cdot \Vert_{\infty,d}$ stands for $\Vert\cdot\Vert_{L^\infty(0,T;L^d(\Omega))}$. We note $D:=\frac{C}{(1+s)\sqrt{2}} [g^\prime]_s$ and $e_k=c(y_k)E(y_k,f_k)^{s/2}$ with  $c(y):= D  d(y)^{1+s}$ and $d(y):=C  e^{C\Vert g^{\prime}(y) \Vert^2_{\infty,d}}$. \eqref{estek} then reads 
\begin{equation}
\label{ekgk}
\sqrt{E(y_{k+1},f_{k+1})}  \leq   \min_{\lambda\in [0,m]}\big(\vert 1-\lambda\vert + \lambda^{1+s} e_k\big)\sqrt{E(y_k,f_k)}.
\end{equation}
We write $\vert g^\prime(y_k)-g^\prime(y_k-\lambda_k Y_k^1) \vert \leq [g^\prime]_s \vert \lambda_k Y_k^1  \vert^s$ so that 
 $$
 \Vert g^{\prime}(y_{k+1})\Vert^2_{L^\infty(0,T;L^d(\Omega))} \leq \Vert g^{\prime}(y_k)\Vert^2_{\infty,d} + \big([g^\prime]_s \lambda_k^s\Vert (Y_k^1)^s\Vert_{\infty,d}\big)^2+2\Vert g^{\prime}(y_k)\Vert_{\infty,d}[g^\prime]_s \lambda_k^{s}\Vert (Y_k^1)^s\Vert_{\infty,d}
 $$
 and 
 $$
 e^{C \Vert g^{\prime}(y_{k+1})\Vert^2_{\infty,d}}\leq e^{C \Vert g^{\prime}(y_k)\Vert^2_{\infty,d}}e^{C \big([g^\prime]_s \lambda_k^{s}\Vert (Y_k^1)^s\Vert_{\infty,d}\big)^2}e^{2C \Vert g^{\prime}(y_k)\Vert_{\infty,d}\big([g^\prime]_s \lambda_k^{s}\Vert (Y_k^1)^s\Vert_{\infty,d}\big)}
 $$
 leading to 
 $$
 \frac{c(y_{k+1})}{c(y_k)}\leq \biggl(e^{C \big([g^\prime]_s \lambda_k^{s}\Vert (Y_k^1)^s\Vert_{\infty,d}\big)^2}e^{2C \Vert g^{\prime}(y_k)\Vert_{\infty,d}\big([g^\prime]_s \lambda_k^{s}\Vert (Y_k^1)^s\Vert_{\infty,d}\big)}\biggr)^{1+s}.
 $$
 We infer that $\Vert  (Y_k^1)^s\Vert_{\infty,d} =\Vert Y_k^1\Vert^s_{\infty,sd}$. Moreover, \eqref{estimateF1Y1} leads to 
$$
\begin{aligned}
\Vert Y_k^1\Vert^s_{\infty,sd}& \leq d^s(y_k)E(y_k,f_k)^{s/2}=\frac{c(y_k)^{\frac{s}{1+s}}}{D^{\frac{s}{1+s}}}E(y_k,f_k)^{s/2}\leq D^{-\frac{s}{1+s}} c(y_k)E(y_k,f_k)^{s/2}\end{aligned}
$$
using that $c(y_k)\geq1$ (by increasing the constant $C$ is necessary). Consequently, 
$$
e^{C \big([g^\prime]_s \lambda ^{s}\Vert (Y_k^1)^s\Vert_{\infty,d}\big)^2}\leq e^{C \big([g^\prime]_s \lambda ^{s}D^{-\frac{s}{1+s}}e_k\big)^2}:=e^{C_1 e_k^2}.
$$
Similarly, 
$$
\begin{aligned}
\Vert g^{\prime}(y_k)\Vert_{\infty,d}\Vert (Y_k^1)^s\Vert_{\infty,d}&\leq \Vert g^{\prime}(y_k)\Vert_{\infty,d} d^s(y_k)E(y_k,f_k)^{s/2}\\
&\leq \Vert g^{\prime}(y_k)\Vert_{\infty,d}\biggl( C  e^{C\Vert g^{\prime}(y) \Vert^2_{L^\infty(0,T;L^d(\Omega))}}\biggr)^sE(y_k,f_k)^{s/2}\\
&\leq \biggl( C  e^{C\Vert g^{\prime}(y) \Vert^2_{L^\infty(0,T;L^d(\Omega))}}\biggr)^{s+1}E(y_k,f_k)^{s/2}\leq \frac{c(y_k)}{D}E(y_k,f_k)^{s/2}=\frac{e_k}{D}\\
\end{aligned}
$$
using that $a\leq C e^{Ca^2}$ for all $a\geq 0$ and $C>0$ large enough. It follows that  
$$
e^{2C \Vert g^{\prime}(y_k)\Vert_{\infty,d}\big([g^\prime]_s \lambda_k^{s}\Vert (Y_k^1)^s\Vert_{\infty,d}\big)}\leq e^{2C [g^\prime]_s \lambda_k^s\frac{e_k}{D}}:=e^{C_2 e_k}
$$
and then $\frac{c(y_{k+1})}{c(y_k)}\leq (e^{C_1e_k^2+C_2 e_k})^{1+s}$. By multiplying \eqref{ekgk} by $c(y_{k+1})$, we obtain the inequality 
\begin{equation}
 \nonumber
e_{k+1}  \leq   \min_{\lambda\in [0,m]} \big(\vert 1-\lambda\vert +e_k \lambda^{1+s}\big) \quad (e^{C_1e_k^2+C_2 e_k})^{1+s}\, e_k.
\end{equation}
If $2e_k<1$, the minimum is reached for $\lambda=1$ leading $\frac{e_{k+1}}{e_k}\leq  e_k (e^{C_1e_k^2+C_2 e_k})^{1+s}$. Consequently, if the initial guess $(y_0,f_0)$ belongs to the set $\{(y_0,f_0)\in \mathcal{A}, e_0< 1/2, e_0 (e^{C_1e_0^2+C_2 e_0})^{1+s}<1\}$,
the sequence $(e_k)_{k>0}$ goes to zero as $k\to \infty$. Since $c(y_k)\geq 1$ for all $k\in \mathbb{N}$, this implies that the sequence $(E(y_k,f_k))_{k>0}$ goes to zero as well. Moreover, from \eqref{estimateF1Y1}, we get $D \Vert (Y_k^1,F_k^1)\Vert_{\mathcal{H}}\leq e_k \sqrt{E(y_k,f_k)}$ and repeating the arguments of the proof of Proposition \ref{prop_gprime_bounded}, we conclude that the sequence $(y_k,f_k)_{k>0}$ converges to a controlled pair for \eqref{eq:wave-NL}.
\end{proof}

These computations does not use the assumption \ref{growth_condition} on the nonlinearity $g$. 
However, the smallness assumption on $e_0$ requires a smallness assumption on $E(y_0,f_0)$ (since $c(y_0)>1$). This is equivalent to assume the controllability of \eqref{eq:wave-NL}. Alternatively, in the case $g(0)=0$, the smallness assumption on $E(y_0,f_0)$ is achieved as soon as $\Vert (u_0,u_1)\Vert_{\boldsymbol{V}}$ is small enough. Therefore, the convergence result stated in Proposition \ref{prop_smallness} is equivalent to the local controllability property for \eqref{eq:wave-NL}. Proposition \ref{prop_smallness} can actually be seen as a consequence of the usual convergence of the Newton method: when $E(y_0,f_0)$ is small enough, i.e., when the initialization is close enough to the solution, then $\lambda_k=1$ for every $k\in\N$ and we recover the standard Newton method.

\paragraph{Weakening of the condition \ref{constraint_g_holder}.}

Given any $s\in [0,1]$, we introduce for any $g\in C^1(\R)$ the following hypothesis : 
\begin{enumerate}[label=$\bf (\overline{H}^{\prime}_s)$,leftmargin=1.5cm]
\item\label{constraint_g_holderbis}\ There exist $\overline{\alpha},\overline{\beta},\gamma\in \R^+$ such that 
$\vert g^\prime(a)-g^\prime(b)\vert \leq \vert a-b\vert^s \big(\overline{\alpha}+\overline{\beta}(\vert a \vert^\gamma +\vert b\vert^\gamma)\big), \quad \forall a,b\in \R$
\end{enumerate}
which coincides with \ref{constraint_g_holder} if $\gamma=0$ for $\overline{\alpha}+\overline{\beta}=[g^\prime]_s$. If $\gamma\in (0,1)$ is small enough and related to the constant $\beta$ appearing in the growth condition \ref{growth_condition}, Theorem \ref{main_theorem} still holds if \ref{constraint_g_holder} is replaced by the weaker hypothesis \ref{constraint_g_holderbis}. Precisely, if $g$ satisfies \ref{growth_condition} and \ref{constraint_g_holderbis} for some $s\in (0,1]$, then the sequence $(y_k,f_k)_{k\in \mathbb{N}}$ defined by \eqref{algo_LS_Y} fulfills the estimate 
$$
E(y_{k+1},f_{k+1})\leq   E(y_k,f_k)\min_{\lambda\in [0,m]}\biggl(\vert 1-\lambda\vert +\lambda^{1+s}\,c(y_k) E(y_k,f_k)^{s/2}\biggr)^2 \nonumber
$$
with
$c(y):= \frac{1}{(1+s)\sqrt{2}}\biggl( \big(\overline{\alpha}+2\overline{\beta}\Vert y_k\Vert^\gamma_{\infty,6\gamma})+\overline{\beta}m^\gamma d(y)^{\gamma}E(y_0,f_0)^\frac{\gamma}{2}\biggr)d(y)^{1+s}$ and  $d(y):=C  e^{C\Vert g^{\prime}(y) \Vert^2_{L^\infty(0,T;L^d(\Omega))}}$. Using Lemma \ref{estimate_Cobs}
with $p^\star=6\gamma\leq  6$ and proceeding as in the proof of Theorem \ref{main_theorem}, one may prove by induction that the sequence $(\Vert y_k\Vert_{L^\infty(0,T; L^6(\Omega))})_{k\in \mathbb{N}}$ is uniformly bounded under the condition $\frac{\gamma+(2 C\beta^2)(1+2s)}{s}<1$
and then deduce the convergence of the sequence $(y_k,f_k)_{k\in \mathbb{N}}$.

\appendix

\section{Appendix: controllability results for the linearized wave equation}\label{sec:linearizedwave}

We recall in this section some \textit{a priori} estimates for the linearized wave equation with potential in $L^\infty(0,T;L^d(\Omega))$ and right hand side in $L^2(Q_T)$. We first recall the crucial observability type estimate proved in \cite[Theorem 2.2]{Zhang2007} (see also \cite[Theorem 2.1]{Li_Zhang_2000}).

\begin{prop}\label{iobs_zhang_2007}\cite{Zhang2007}
Assume that $\omega$ and $T$ satisfy the assumptions of Theorem \ref{ZhangTH}. For any $A\in L^\infty(0,T;L^d(\Omega))$, and $(\phi_0,\phi_1)\in \boldsymbol{H}:=L^2(\Omega)\times H^{-1}(\Omega)$, the weak solution $\phi$ of 
\begin{equation}
\label{eq:wave-phi}
\left\{
\begin{aligned}
& \partial_{tt}\phi - \Delta \phi + A\phi = 0, & \textrm{in}\,\, Q_T,\\
& \phi=0, & \textrm{on}\,\, \Sigma_T, \\
& (\phi(\cdot,0),\phi_t(\cdot,0))=(\phi_0,\phi_1), & \textrm{in}\,\, \Omega,
\end{aligned}
\right.
\end{equation}
satisfies the observability inequality $\Vert \phi_0,\phi_1\Vert_{\boldsymbol{H}} \leq C e^{C \Vert A\Vert^2_{L^\infty(0,T;L^d(\Omega))}} \Vert \phi\Vert_{2,q_T}$
for some $C>0$ only depending on $\Omega$ and $T$.
\end{prop}
\par\noindent
Classical arguments then lead to following controllability result.  
\begin{prop}\label{controllability_result}\cite{Zhang2007}
Let $A\in L^{\infty}(0,T;L^d(\Omega))$, $B\in L^2(Q_T)$ and $(z_0,z_1)\in \boldsymbol{V}$. Assume that $\omega$ and $T$ satisfy the assumptions of Theorem \ref{ZhangTH}. There exists a control function $u\in L^2(q_T)$ such that the solution of 
\begin{equation}
\label{wave_z}
\left\{
\begin{aligned}
& \partial_{tt} z- \Delta z +  A z  = u 1_{\omega} + B, & \textrm{in}\,\, Q_T,\\
& z=0, & \textrm{on}\,\, \Sigma_T, \\
& (z(\cdot,0),z_t(\cdot,0))=(z_0,z_1), &\textrm{in}\,\, \Omega,
\end{aligned}
\right.
\end{equation}
satisfies $(z(\cdot,T),z_t(\cdot,T))=(0,0)$ in $\Omega$.  Moreover, the unique pair $(u,z)$ of minimal control norm satisfies
\begin{equation}\label{estimate_u_z}
\Vert u\Vert_{2,q_T} + \Vert (z,\partial_t z)\Vert_{L^\infty(0,T;\boldsymbol{V})}\leq C\biggl(\Vert B\Vert_2+ \Vert z_0,z_1\Vert_{\boldsymbol{V}}\biggr) e^{C \Vert A\Vert^2_{L^\infty(0,T;L^d(\Omega))}} 
\end{equation}
for some constant $C>0$ only depending on $\Omega$ and $T$. 
\end{prop}

Let $p^\star\in \mathbb{N}^\star$ such that $p^\star<\infty$ if $d=2$ and $p^\star<6$ if $d=3$. We next discuss some properties of the operator $K: L^\infty(0,T;L^{p^\star}(\Omega))\to L^\infty(0,T;L^{p^\star}(\Omega))$ defined by $K(\xi)=y_{\xi}$, a null controlled solution of the linear boundary value problem \eqref{NL_z} with the control $f_{\xi}$ of minimal $L^2(q_T)$ norm. Proposition \ref{controllability_result} with $B=-g(0)$ gives
\begin{equation}
\label{estimateK}
\Vert (y_{\xi},\partial_t y_{\xi})\Vert_{L^\infty(0,T;\boldsymbol{V})}\leq C\Big(\Vert u_0,u_1\Vert_{\boldsymbol{V}}+\Vert g(0)\Vert_2  \Big)e^{C\Vert \widehat{g}(\xi)\Vert^2_{L^\infty(0,T;L^d(\Omega))}}
\end{equation}
where the function $\hat{g}$ is defined in \eqref{NL_z}. We assume that $g\in C^1(\mathbb{R})$ satisfies the following asymptotic condition (slightly weaker than  \ref{asymptotic_behavior}): there exists a $\overline{\beta}$ small enough
such that $\limsup_{\vert r\vert \to \infty} \frac{\vert g(r)\vert }{\vert r\vert \ln^{1/2}\vert r\vert}\leq \overline{\beta}$, i.e. 

\begin{enumerate}[label=$\bf (H^\prime_1)$,leftmargin=1.5cm]
\item\label{growth_conditionp}\ There exist $\overline{\alpha}\geq 0$ and $\overline{\beta}\geq 0$ small enough such that $\vert g(r)\vert\leq \overline{\alpha} + \overline{\beta} (1+\vert r\vert))\ln^{1/2}(1+\vert r\vert )$ for every $r$ in $\R$. % \label{asymptotic_behavior}
\end{enumerate}

This implies that $\hat{g}$ satisfies $\vert \widehat{g}(r)\vert \leq \overline{\alpha} + \overline{\beta} \ln^{1/2}(1+\vert r\vert)$ for every $r\in \mathbb{R}$ and some constant $\overline{\alpha}>0$. This also implies that $\widehat{g}(\xi)\in L^\infty(0,T; L^d(\Omega))$ for any $\xi\in L^\infty(0,T;L^{p^\star}(\Omega))$. Assuming $2C\overline{\beta}^2\leq 1$ and proceeding as in the proof of Lemma \ref{estimate_Cobs}, we get 

$$
e^{C \Vert \widehat{g}(\xi)\Vert^2_{L^\infty(0,T;L^d(\Omega))}}\leq C_1\biggl(1+\frac{\Vert \xi\Vert_{L^\infty(0,T;L^{p^\star}(\Omega))}}{\vert \Omega\vert^{1/p^{\star}}} \biggr)^{2 C\overline{\beta}^2}, \quad \forall \xi\in L^\infty(0,T;L^{p^\star}(\Omega))
$$
for some $C_1=C_1(\alpha)$.
Using \eqref{estimateK}, we then infer that
$$
 \Vert y_{\xi}\Vert_{L^\infty(0,T;L^{p^\star}(\Omega))}\leq C\Big(\Vert u_0,u_1\Vert_{\boldsymbol{V}}+\Vert g(0)\Vert_2\Big) C_1 \biggl(1+\frac{\Vert \xi\Vert_{{L^\infty(0,T;L^{p^\star}(\Omega))}}}{\vert\Omega\vert^{1/p^\star}}\biggr)^{2C \overline{\beta}^2}, \quad \forall \xi\in L^\infty(0,T;L^{p^\star}(\Omega)).
$$
Taking $\overline{\beta}$ small enough so that $2C\overline{\beta}^2<1$, we conclude that there exists $M>0$ such that $\Vert \xi\Vert_{L^\infty(0,T;L^{p^\star}(\Omega))}\leq M$ implies $\Vert K(\xi)\Vert_{L^\infty(0,T;L^{p^\star}(\Omega))}\leq M$. This is the argument (introduced in \cite{zuazua93} for the one dimensional case and) implicitly used in \cite{Li_Zhang_2000} to prove the controllability of \eqref{eq:wave-NL}. Note that, in contrast to $\overline{\beta}$, $M$ depends on $\Vert u_0,u_1\Vert_{\boldsymbol{V}}$ (and increases with $\Vert u_0,u_1\Vert_{\boldsymbol{V}}$).

The following result gives an estimate of the difference of two controlled solutions. %We shall also use the following (regular) perturbation controllability result. 

\begin{lemma}\label{gap}
Let $A\in L^\infty(0,T;L^d(\Omega))$, $a\in L^\infty(0,T;L^{d+\epsilon}(\Omega))$ for any $\epsilon>0$, $B\in L^2(Q_T)$ and $(u_0,u_1)\in \boldsymbol{V}$. Let $u$ and $v$ be the null controls of minimal $L^2(q_T)$ norm for $y$ and $z$ respectively solutions of 
\begin{equation}
\label{wave_y}
\left\{
\begin{aligned}
& \partial_{tt}y - \Delta y +  A y  = u 1_{\omega} + B & \textrm{in}\  Q_T,\\
& y=0 & \textrm{on}\  \Sigma_T, \\
& (y(\cdot,0),\partial_t y(\cdot,0))=(u_0,u_1) &\textrm{in}\  \Omega,
\end{aligned}
\right.
\end{equation}
and 
\begin{equation}
\label{wave_zz}
\left\{
\begin{aligned}
& \partial_{tt}z -  \Delta z +  (A+a)z   = v 1_{\omega} + B & \textrm{in}\  Q_T,\\
& z=0 & \textrm{on}\  \Sigma_T, \\
& (z(\cdot,0),\partial_t z(\cdot,0))=(u_0,u_1) &\textrm{in}\  \Omega .
\end{aligned}
\right.
\end{equation}
Then, 
$$
\begin{aligned}
\Vert y-z\Vert_{L^\infty(0,T;H_0^1(\Omega))} 
& \leq C \Vert a \Vert_{L^\infty(0,T;L^{d+\epsilon}(\Omega))} \big(\Vert B\Vert_2+ \Vert u_0,u_1\Vert_{\boldsymbol{V}}\big) e^{C\Vert A+a\Vert^2_{L^\infty(0,T;L^d(\Omega))}}e^{C\Vert A\Vert^2_{L^\infty(0,T;L^d(\Omega))}}
\end{aligned}
$$
for some constant $C>0$ only depending on $\Omega$ and $T$.
\end{lemma}

\begin{proof} 
The controls of minimal $L^2(q_T)$ norm for $y$ and $z$ are given by $u=\phi 1_\omega$ and 
$v=\phi_a 1_\omega$ where $\phi$ and $\phi_a$ respectively solve the adjoint equations 
\begin{equation*}
\left\{
\begin{aligned}
& \partial_{tt}\phi - \Delta\phi +  A \phi   = 0 & \textrm{in}\  Q_T,\\
& \phi=0 & \textrm{on}\  \Sigma_T, \\
& (\phi(\cdot,0),\partial_t\phi(\cdot,0))=(\phi_0,\phi_1) &\textrm{in}\  \Omega,
\end{aligned}
\right.
\qquad\quad 
\left\{
\begin{aligned}
& \partial_{tt}\phi_{a} - \Delta\phi_{a} +  (A+a) \phi_a   = 0 & \textrm{in}\  Q_T,\\
& \phi=0 & \textrm{on}\  \Sigma_T, \\
& (\phi(\cdot,0),\partial_t\phi(\cdot,0))=(\phi_{a,0},\phi_{a,1}) &\textrm{in}\  \Omega,
\end{aligned}
\right.
\end{equation*}
for some appropriate $(\phi_0,\phi_1), (\phi_{a,0},\phi_{a,1})\in \boldsymbol{H}$. In particular, $\phi,\phi_a\in C([0,T];L^2(\Omega))\cap C^1([0,T]; H^{-1}(\Omega))$. Hence $Z:=z-y$ solves 
\begin{equation}
\label{diffZ}
\left\{
\begin{aligned}
& \partial_{tt}Z - \Delta Z +  (A+a)Z   = \Phi 1_{\omega} -ay & \textrm{in}\  Q_T,\\
& Z=0 & \textrm{on}\  \Sigma_T, \\
& (Z(\cdot,0), \partial_t z(\cdot,0))=(0,0) &\textrm{in}\  \Omega,
\end{aligned}
\right.
\end{equation}
and $\Phi:=\phi_a-\phi$ solves
\begin{equation}
 \nonumber
\left\{
\begin{aligned}
& \partial_{tt}\Phi - \Delta\Phi +  (A+a)\Phi  = -a\phi & \textrm{in}\  Q_T,\\
& \Phi=0 & \textrm{on}\  \Sigma_T, \\
& (\Phi(\cdot,0), \partial_t\Phi(\cdot,0))=(\phi_{a,0}-\phi_0,\phi_{a,1}-\phi_1) &\textrm{in}\  \Omega.
\end{aligned}
\right.
\end{equation}
In particular  (since $a\in L^\infty(0,T;L^{d+\epsilon}(\Omega))$ and $\phi\in L^\infty(0,T;L^2(\Omega))$), we get that $a\phi\in L^\infty(0,T;H^{-1}(\Omega))$ and therefore $(\Phi,\Phi_t)\in C([0,T];\boldsymbol{H})$, see \cite[Theorem 2.3]{LasieckaLionsTriggiani86}. We decompose $\Phi:=\Psi+\psi$ where $\Psi$ and $\psi$ solve respectively 
\begin{equation}
 \nonumber
\left\{
\begin{aligned}
& \partial_{tt}\Psi - \Delta\Psi +  (A+a)\Psi  = 0 & \textrm{in}\  Q_T,\\
& \Psi=0 & \textrm{on}\  \Sigma_T, \\
& (\Psi(\cdot,0), \partial_t\Psi(\cdot,0))=(\phi_{a,0}-\phi_0,\phi_{a,1}-\phi_1) &\textrm{in}\  \Omega,
\end{aligned}
\right.
\qquad \left\{
\begin{aligned}
& \partial_{tt}\psi - \Delta\psi +  (A+a)\psi   =  -a\phi & \textrm{in}\  Q_T,\\
& \psi=0 & \textrm{on}\  \Sigma_T, \\
& \psi(\cdot,0), \partial_t\psi(\cdot,0))=(0,0) &\textrm{in}\  \Omega,
\end{aligned}
\right.
\end{equation}
and we deduce that $\Psi1_\omega$ is the control of minimal $L^2(q_T)$ norm for $Z$ solution of  
\begin{equation}
 \nonumber
\left\{
\begin{aligned}
& \partial_{tt}Z - \Delta Z +  (A+a)Z   = \Psi 1_{\omega} + \Big(\psi 1_\omega-ay\Big) & \textrm{in}\  Q_T,\\
& Z=0 & \textrm{on}\  \Sigma_T, \\
& (Z(\cdot,0), \partial_t Z(\cdot,0))=(0,0) &\textrm{in}\  \Omega.
\end{aligned}
\right.
\end{equation}
Proposition \ref{controllability_result} implies that 
\begin{equation}\nonumber
\Vert \Psi\Vert_{2,q_T}+ \Vert (Z,\partial_t Z)\Vert_{L^\infty(0,T;\boldsymbol{V})}\leq C \Vert \psi 1_\omega-ay\Vert_2 e^{C\Vert A+a\Vert^2_{L^\infty(0,T;L^d(\Omega))}}.
\end{equation}

Moreover, \cite[Lemma 2.4]{Li_Zhang_2000} applied to $\psi$ leads to 
$$
\Vert (\psi,\psi_t)\Vert_{L^\infty(0,T;\boldsymbol{H})}\leq \Vert a\varphi\Vert_{L^\infty(0,T; H^{-1}(\Omega))}e^{C \Vert A+a\Vert_{L^\infty(0,T;L^d(\Omega))}}
$$ 
and $\Vert \psi\Vert_{L^2(q_T)}\leq C\Vert a\Vert_{L^\infty(0,T; L^{d+\epsilon}(\Omega))} \Vert\phi\Vert _{2}e^{C\Vert A+a\Vert^2_{L^\infty(0,T;L^d(\Omega))}}$. But, using again \cite[Lemma 2.4]{Li_Zhang_2000}, we infer that $\Vert\phi\Vert _2\leq C \Vert \phi_0,\phi_1\Vert_{\boldsymbol{H}}e^{C\Vert A\Vert_{L^\infty(0,T;L^d(\Omega))}}$ while \cite[Theorem 2.1]{Li_Zhang_2000} gives $\Vert \phi_0,\phi_1\Vert_{\boldsymbol{H}}\leq C e^{C\Vert A\Vert^2_{L^\infty(0,T;L^d(\Omega))}}\Vert \phi\Vert_{2,q_T}$. Since $u=\phi 1_\omega$, we obtain $\Vert\phi\Vert _2\leq C^2e^{C\Vert A\Vert^2_{L^\infty(0,T;L^d(\Omega))}}\Vert u\Vert_{2,q_T}e^{C\Vert A\Vert_{L^\infty(0,T;L^d(\Omega))}}$ and then   
$$
\Vert \psi\Vert_{L^2(q_T)}\leq C\Vert a\Vert_{L^\infty(0,T; L^{d+\epsilon}(\Omega))}e^{C\Vert A\Vert^2_{L^\infty(0,T;L^d(\Omega))}} e^{C\Vert A+a\Vert^2_{L^\infty(0,T;L^d(\Omega))}} \Vert u\Vert _{2,q_T}
$$
from which we deduce that
$$
\begin{aligned}
\Vert Z\Vert_{L^\infty(0,T;H_0^1(\Omega))} & \leq C \Big(\Vert \psi\Vert_{2,q_T}+ \Vert a\Vert_{L^\infty(0,T;L^{d+\epsilon}(\Omega))}\Vert y\Vert_2\Big) e^{C\Vert A+a\Vert^2_{L^\infty(0,T;L^d(\Omega))}}\\
& \leq C \Vert a \Vert_{L^\infty(0,T;L^{d+\epsilon}(\Omega))} \biggl(\Vert B\Vert_2+ \Vert u_0,u_1\Vert_{\boldsymbol{V}}\biggr) e^{C\Vert A+a\Vert^2_{L^\infty(0,T;L^d(\Omega))}}e^{C\Vert A\Vert^2_{L^\infty(0,T;L^d(\Omega))}}
\end{aligned}
$$
leading to the result. 
\end{proof}

This result allows to establish the following property for the operator $K$.

\begin{lemma}\label{lem_final}
Under the assumptions done in Theorem \ref{ZhangTH}, let $M=M(\Vert u_0, u_1\Vert_{\boldsymbol{V}},\overline{\beta})$ be such that $K$ maps $B_{L^\infty(0,T;L^{d+\epsilon}(\Omega))}(0,M)$ into itself and assume that $\hat{g}^{\prime}\in L^\infty(\mathbb{R})$.
For any $\xi^i\in B_{L^\infty(0,T;L^{d+\epsilon}(\Omega))}(0,M)$, $i=1,2$, there exists $c(M)>0$ such that  
\begin{equation}
 \nonumber
\Vert K(\xi^2)-K(\xi^1)\Vert_{L^\infty(0,T;H_0^1(\Omega))}\leq c(M) \Vert \hat{g}^{\prime}\Vert_{\infty} \Vert \xi^2-\xi^1\Vert_{L^\infty(0,T;L^{d+\epsilon}(\Omega))}.
\end{equation}
\end{lemma}

\begin{proof} For any $\xi^i\in B_{L^\infty(0,T;L^{p^\star}(\Omega))}(0,M)$, $i=1,2$, let $y_{\xi^i}=K(\xi^i)$ be the null controlled solution of 
\begin{equation}
 \nonumber
%\label{NL_z_i}
\left\{
\begin{aligned}
& \partial_{tt}y_{\xi^i} - \Delta y_{\xi^i} +  y_{\xi^i} \,\widehat{g}(\xi^i)= -g(0)+f_{\xi^i} 1_{\omega} &\textrm{in}\  Q_T,\\
& y_{\xi^i}=0 &\textrm{on}\  \Sigma_T, \\
& (y_{\xi^i}(\cdot,0),\partial_ty_{\xi^i}(\cdot,0))=(u_0,u_1) &\textrm{in}\  \Omega,
\end{aligned}
\right.
\end{equation}
with the control $f_{\xi^i} 1_{\omega}$ of minimal $L^2(q_T)$ norm. We observe that $y_{\xi^2}$ is solution of 
\begin{equation}
 \nonumber
%\label{NL_z_i}
\left\{
\begin{aligned}
& \partial_{tt}y_{\xi^2} - \Delta y_{\xi^2} +  y_{\xi^2} \,\widehat{g}(\xi^1)+ y_{\xi^2}(\widehat{g}(\xi^2)-\widehat{g}(\xi^1))= -g(0)+f_{\xi^2} 1_{\omega} &\textrm{in}\  Q_T,\\
& y_{\xi^2}=0 &\textrm{on}\  \Sigma_T, \\
& (y_{\xi^2}(\cdot,0),\partial_t y_{\xi^2}(\cdot,0))=(u_0,u_1) &\textrm{in}\  \Omega.
\end{aligned}
\right.
\end{equation}
It follows from Lemma \ref{gap} applied with $B=-g(0)$, $A=\hat{g}(\xi^1)$, $a=\hat{g}(\xi^2)-\hat{g}(\xi^1)$, that 
\begin{equation}\label{gapxi}
\Vert y_{\xi^2}-y_{\xi^1} \Vert_{L^\infty(0,T;H_0^1(\Omega))}\leq A(\xi^1,\xi^2)\Vert \widehat{g}(\xi^2)-\widehat{g}(\xi^1)\Vert_{L^\infty(0,T;L^{d+\epsilon}(\Omega))}
\end{equation}
where the positive constant  
$$
A(\xi^1,\xi^2):=C  \biggl(\Vert g(0)\Vert_2+ \Vert u_0,u_1\Vert_{\boldsymbol{V}}\biggr) e^{C\Vert \hat{g}(\xi^1)\Vert^2_{L^\infty(0,T;L^d(\Omega))}}e^{C\Vert \hat{g}(\xi^2)\Vert^2_{L^\infty(0,T;L^d(\Omega))}}
$$
is bounded by some $c(M)>0$ for every $\xi^i\in B_{L^\infty(0,T;L^{d+\epsilon}(\Omega))}(0,M)$. The result follows from \eqref{gapxi}. 
\end{proof}

\begin{remark}\label{rem_contract}
By Lemma \ref{lem_final}, if $\Vert \hat{g}^{\prime}\Vert_{\infty}<  1/c(M)$ then the operator $K:L^\infty(0,T;L^{d+\epsilon}(\Omega))\to L^\infty(0,T;L^{d+\epsilon}(\Omega))$ is contracting. Note however that the bound depends on the norm $\Vert u_0,u_1\Vert_{\boldsymbol{V}}$ of the initial data to be controlled.
\end{remark}

 \bibliographystyle{siam}
 %\bibliography{heatLS.bib}
 \bibliography{wavecontinuation.bib}
\end{document}